\documentclass[11pt]{article}
\usepackage[margin=1in]{geometry}
\usepackage{graphicx,xcolor,cite,mathtools}
\usepackage{amssymb,amsmath,amsthm,mathrsfs,paralist,bm,esint,setspace}
\RequirePackage[colorlinks,citecolor=blue,urlcolor=blue]{hyperref}
\usepackage{marginnote}

\newcommand{\Z}{{\mathbb{Z}}}
\newcommand{\R}{{\mathbb{R}}}
\newcommand{\E}{\mathrm{E}}
\newcommand{\HH}{\mathcal{H}}
\renewcommand{\P}{\mathrm{P}}
\renewcommand{\d}{\mathrm{d}}
\newcommand{\<}{\langle}
\renewcommand{\>}{\rangle}
\newcommand{\e}{\mathrm{e}}
\newcommand{\Var}{\text{\rm Var}}

\DeclareMathOperator{\Cov}{\text{\rm Cov}}

\input cyracc.def

\title{Gaussian fluctuation for spatial average of parabolic Anderson model with Neumann/Dirichlet/periodic boundary conditions\thanks{%
	Research supported in part by   FNR grant APOGee (R-AGR-3585-10) at University of Luxembourg.}}
\author{
				Fei Pu\\University of Luxembourg\\\texttt{fei.pu@uni.lu}
	}
\date{\today}

\begin{document}
\newtheorem{stat}{Statement}[section]
\newtheorem{proposition}[stat]{Proposition}
\newtheorem*{prop}{Proposition}
\newtheorem{corollary}[stat]{Corollary}
\newtheorem{theorem}[stat]{Theorem}
\newtheorem{lemma}[stat]{Lemma}
\theoremstyle{definition}
\newtheorem{definition}[stat]{Definition}
\newtheorem*{cremark}{Remark}
\newtheorem{remark}[stat]{Remark}
\newtheorem*{OP}{Open Problem}
\newtheorem{example}[stat]{Example}
\newtheorem{nota}[stat]{Notation}
\numberwithin{equation}{section}

\maketitle

\begin{abstract}
	Consider the
	parabolic Anderson model $\partial_tu=\frac{1}{2}\partial_x^2u+u\, \eta$ on the interval $[0, L]$ with Neumann, Dirichlet or periodic boundary conditions, driven by space-time white noise $\eta$. 
	Using Malliavin-Stein method, we establish the central limit theorem for the fluctuation of the spatial integral $\int_0^Lu(t\,, x)\, \d x$ as $L$ tends to infinity, where the limiting Gaussian distribution is independent of the choice of the boundary conditions and coincides with the Gaussian fluctuation for the spatial average of parabolic Anderson model on the whole space $\R$.
\end{abstract}

\bigskip

 \smallskip

\smallskip
	
\noindent{\it Running head:} CLT for PAM on interval.

{
}

\section{Introduction}

Consider the {\em parabolic Anderson model} on the interval $[0, L]$
\begin{align}\label{PAM}
  \begin{cases}
	\partial_t u(t\,, x) = \tfrac12\partial_x^2 u(t\,, x) + u(t\,, x)\,\eta(t \,, x),\qquad 0<t \leq T, \,\, x \in [0, L] ,\\
	 u(0)\equiv 1,
	\end{cases}
\end{align}
subject to Neumann, Dirichlet or periodic boundary conditions, where $T>0$ is fixed and $\eta$ denotes space-time white noise on 
$[0, T]\times\R$, which is a generalized centered Gaussian process with covariance given by
\begin{align*}
\E[\eta(t\,, x)\eta(s\,, y)]= \delta_0(t-s)\delta_0(x-y), \quad t, s\in  [0, T], \, x, y \in \R.
\end{align*}
Following Walsh\cite{Walsh},  the mild solution to the stochastic PDE \eqref{PAM} satisfies 
the following integral equation:
\begin{align}\label{mild}
	u(t\, , x) = \int_0^LG_{t}(x, y)\, \d y+ \int_0^t\int_0^LG_{t - s}(x, y)u(s\,, y)\,\eta(\d s\,\d y),
\end{align}
where $G_t(x, y)$ denotes the heat kernel on $[0, L]$ with Neumann, Dirichlet or periodic boundary conditions, with expressions given by 
\eqref{heatkernel1}, \eqref{heatkernel2} or \eqref{Pheat}, respectively. Here, we omit the dependence on the parameter $L$ of the solution $u$ and the heat kernel $G$ to simplify the notation. 

In order to study the Gaussian fluctuation of the spatial average of the solution, we introduce
\begin{equation} \label{average}
	\mathcal{S}_{L,t} := \frac 1L \int_0^L \{ u(t\,,x) -\E[u(t\,, x)]\} \,\d x
	\qquad\text{for all $L\geq 1$ and $t\ge0$}.
\end{equation}

The goal of this paper is to prove the following  central limit theorem.

\begin{theorem}\label{th:FCLT}
         Fix $T>0$. 
	Then, in all the three cases of boundary conditions, as $L\to\infty$,
	\begin{equation}\label{FCLT}
		\sqrt{L}\,
		\mathcal{S}_{L,\bullet}		
		\xrightarrow{C[0,T]} \int_0^{\bullet}\sqrt{f(t)}\, \d B_t,
	\end{equation}
	where 
	\begin{align}\label{f}
	f(t)=2\e^{t/4} \int_{-\infty}^{\sqrt{t/2}}\frac{1}{\sqrt{2\pi}}\e^{-\frac{y^2}{2}}\d y, \quad t\geq 0
	\end{align}
	and $B$ denotes a standard one-dimensional Brownian motion,
	and ``$\xrightarrow{C[0,T]}$'' denotes the convergence in law in the 
	space of continuous functions $C[0\,,T]$.
\end{theorem}

It is well known that strong mixing together with  a standard blocking argument can imply a CLT
(see Bradley \cite{Bradley}). However, it is not easy to determine the conditions under which the strong mixing holds in the context of SPDEs. Recently, Chen et al \cite{CKNP_b} have introduced a method to study spatial CLT for a large class of SPDEs, based on Malliavin calculus, Poincar\'e inequalites, compactness arguments and Paul L\'evy's characterization theorem of Brownian motion. This method has been generalized in \cite{CKNP_c} and adapted to study the CLT for infinitely-many interacting diffusion processes.

The proceeding two approaches to CLT require stationarity of the process and unfortunately they  do not apply to our case since the solution $\{u(t\,, x): (t, x)\in [0, T]\times [0, L]\}$ to \eqref{PAM} is clearly not stationary in space when we consider the Neumann/Dirichlet boundary conditions.  In order to prove the CLT in Theorem \ref{th:FCLT}, we will appeal to Malliavin-Stein method, which was introduced by Huang et al  \cite{HNV2018} 
for the one-dimensional stochastic heat equation driven by a  space-time white noise,
and  later widely extended to multidimensional SPDEs driven by Gaussian noise in  
\cite{HNVZ2019, DNZ2018, GNZ20, NZ2020, CKNP_d, KNP20}.  This approach to CLT provides a  
convergence rate in terms of  total variation distance, using a combination of Malliavin calculus and 
Stein's method for normal approximations  (see Nourdin and Peccati \cite{NP09,NP}).
Also, as we will see,  the Malliavin-Stein approach to CLT applies in our non-stationary setting.

Recall that  the total variation distance between two random variables $X$ and $Y$ is defined as
\begin{align}\label{tvd:def}
	d_{\rm TV}  (X\,,Y)= \sup_{B\in \mathcal{B}(\R)} | \P(X\in B)- \P(Y\in B)|,
\end{align}
where  $\mathcal{B}(\R)$ denotes the family of all Borel subsets of $\R$. 
We abuse  notation and let $d_{\rm TV}(F\,,{\rm N}(0\,,1))$  denote the 
total variation distance between the law of $F$ and the ${\rm N}(0\,,1)$ law.

 In  the following theorem, we derive the convergence rate for the total variation distance between the normalization of 
$\mathcal{S}_{L,t}$ and standard normal distribution ${\rm N}(0\,,1)$.
\begin{theorem}\label{TVD}
	For every $t>0$ there exists
	a real number $c = c(t) >0$ 
	such that for all $L \geq  1$,
	\begin{align}\label{TVDeq}
		d_{\rm TV} \left(  \frac{ \mathcal{S}_{L,t}}
		{ \sqrt{{\rm Var}(\mathcal{S}_{L,t})}} ~,~
		{\rm N}(0\,,1)\right) \leq  \frac{c }{\sqrt L}.
	\end{align}
\end{theorem}

\begin{remark}
The function $f(t)$ in \eqref{f} is equal to the second moment of solution (at time $t$)  to parabolic Anderson model on $\R$ driven by space-time white noise with constant initial condition.  Indeed, let $\{U(t\,,x ): (t, x)\in \R_+\times \R\}$ solve 
\begin{align}\label{PAM2}
  \begin{cases}
	\partial_t U(t\,, x) = \tfrac12\partial_x^2 U(t\,, x) + U(t\,, x)\,\eta(t \,, x),\qquad  t>0, \,\, x \in \R ,\\
	 U(0)\equiv 1.
	\end{cases}
\end{align}
Then according to \cite[(2.28) and (2.18)]{CD15}, for all $t\geq 0$ and $x\in \R$, 
\begin{align}\label{secondm}
\E[U(t\,, x)^2]=2\e^{t/4} \int_{-\infty}^{\sqrt{t/2}}\frac{1}{\sqrt{2\pi}}\e^{-\frac{y^2}{2}}\d y= f(t).
\end{align}
\end{remark}

\begin{remark}
          The limiting Gaussian process in \eqref{FCLT} coincides with the Gaussian fluctuation of the spatial average of $U$ that solves \eqref{PAM2}.  In fact, as a special case of Huang et al \cite[Theorem 1.2]{HNV2018},  they have proved that as $R\to\infty$,
          \begin{align*}
          \frac{1}{\sqrt{R}}\int_0^R[U(\bullet\,, x) -1] \, \d x	\xrightarrow{C[0,T]} \int_0^{\bullet}\sqrt{\E[U(t\,, 0)^2]}\, \d B_t= 	 \int_0^{\bullet}\sqrt{f(t)}\, \d B_t;
          \end{align*}
  see \eqref{secondm} for the identity.
\end{remark}

We will prove Theorems \ref{th:FCLT} and \ref{TVD} in Section \ref{Sec:TVD} based on the Malliavin-Stein method; see Propositions \ref{pr:MS} and \ref{lemma: NP 6.1.2}.  Here, we point out that unlike the cases considered in the literature mentioned above, in our situation the solution to \eqref{PAM} depends on the length of the interval $L$. We need control the moments of the solution as well as its Malliavin derivative uniformly as $L\to\infty$; see Lemmas \ref{prop:um} and \ref{derivative:estimate} in Section \ref{sec:p}.  Section \ref{asym:cov} is devoted to 
the asymptotic behavior of the covariance as $L\to \infty$, which leads to the expression of the limit Gaussian process in \eqref{FCLT} and the formula of the function $f$ in \eqref{f}. 
And the last section is an Appendix that contains a few technical
lemmas on the heat kernel that are used throughout the paper.

 We write $\|Z\|_k$ instead of
$(\E[|Z|^k])^{1/k}$, for every $Z\in L^k(\Omega)$.

\section{Preliminaries}\label{sec:p}

\subsection{Clark-Ocone formula}
Let $\mathcal{H}= L^2([0, T] \times \R)$.
The Gaussian family $\{ W(h)\}_{h \in \mathcal{H}}$ formed by the Wiener integrals
\[
	W(h)= \int_{[0, T]\times \R}  h(s\,,x)\, \eta(\d s\, \d x)
\]
defines an  {\it isonormal Gaussian process} on the Hilbert space $\mathcal{H}$.
In this framework we can develop the Malliavin calculus (see Nualart \cite{Nualart}).
We denote by $D$ the derivative operator.
Let $\{\mathcal{F}_s\}_{s \geq 0}$ denote the filtration generated by 
the space-time white noise $\eta$. 

We recall the following Clark-Ocone formula  (see Chen et al \cite[Proposition 6.3]{CKNP}):
\[
	F= \E [F]  + \int_{[0, T]\times \R}
	\E\left[D_{s,y} F \mid \mathcal{F}_s\right] \eta(\d s \, \d z)
	\qquad\text{a.s.},
\]
valid for every random variable $F$ in the Gaussian Sobolev space $\mathbb{D}^{1,2}$. 
Thanks to Jensen's inequality for conditional expectations, the above Clark-Ocone formula 
readily yields the following Poincar\'e-type inequality, which plays an important
role throughout the paper:
\begin{equation} \label{Poincare:Cov}
	|\Cov(F\,, G)| \le \int_0^T\d s \int_\R\d z\
	\left\| D_{s,z } F  \right\|_2
	\left\| D_{s,z}G \right\|_2
	\qquad\text{for all $F,G\in\mathbb{D}^{1,2}$.}		
\end{equation}

\subsection{The Malliavin-Stein method}

Recall the total variation distance between two random variables defined in \eqref{tvd:def}.
The following bound on $d_{\rm TV}(F\,,{\rm N}(0\,,1))$
follows from a suitable combination of ideas from the Malliavin calculus and
Stein's method for normal approximations; see Nualart and Nualart \cite[Theorem 8.2.1]{NN}.

\begin{proposition}\label{pr:MS}
	Suppose that $F\in \mathbb{D}^{1,2}$ satisfies $\E(F^2)=1$ and  $F=\delta(v)$ 
	for some $v$ in the $L^2(\Omega)$-domain  of the divergence operator $\delta$. 
	Then, 
	\[
		d_{\rm TV} (F\,,  {\rm N}(0\,,1)) \le  
		2\sqrt{ {\rm Var} \left ( \langle DF\,, v \rangle_{\mathcal{H}} \right) }.
	\]
\end{proposition}

The proof of Theorem \ref{th:FCLT}  is based on
the following  generalization of  a result of Nourdin and Peccati \cite[Theorem 6.1.2]{NP}. 

\begin{proposition}\label{lemma: NP 6.1.2}
	Let $F=( F^{(1)}, \dots, F^{(m)})$ be a random vector such that, for every $i=1,\ldots,m$,
	$F^{(i)} = \delta (v^{(i)})$ for some $v^{(i)} \in {\rm Dom}\, [\delta]$.
	Assume additionally that
	$F^{(i)} \in \mathbb{D}^{1,2}$ for $i=1,\dots,m$. Let $G$ be a centered
	$m$-dimensional Gaussian random vector with covariance matrix $(C_{i,j}) _{1\le i,j \le m} $. 
	Then, for every $h\in C^2(\R^m)$ that has bounded second partial derivatives, 
	\[
		| \E( h(F)) -\E (h(G)) | \le \tfrac 12 \|h ''\|_\infty 
		\sqrt{   \sum_{i,j=1}^m   \E \left( \left|
		C_{i,j} - \langle DF^{(i)}\,, v^{(j)} \rangle_{\HH} \right|^2
		\right)},
	\]
	where 
	\[
		\|h'' \| _\infty := \adjustlimits\max_{1\le i,j \le m} 
		\sup_{x\in\R^m}  \left| \frac { \partial ^2h (x) } {\partial x_i \partial x_j} \right|.
	\]
\end{proposition}

\subsection{Moments and Malliavin derivative of \texorpdfstring{$u(t,x)$}{u(t,x)}}

In this section, we will give some upper bounds on the moments and Malliavin derivative of $u(t, x)$,  uniformly for $L\geq 1$. We first remark that mild form in \eqref{mild} can be understood as 
\begin{align}\label{mild2}
	u(t\, , x) = \int_{\R} \bm{1}_{[0, L]}(y)G_{t}(x, y) \d y+ \int_0^t\int_{\R} \bm{1}_{[0, L]}(y)G_{t - s}(x, y)u(s\,, y)\,\eta(\d s\,\d y). 
\end{align}
This means that for every $L\geq 1$, the solution $\{u(t\,, x): (t, x)\in [0, T] \times [0, L]\}$ can be viewed as a function of the space-time white noise $\eta$ on $[0, T] \times \R$. In what follows, we will alway write the spatial integral $\int_0^L$ instead of $\int_{\R} \bm{1}_{[0, L]}$, as it is clear in the context.

We now define the Picard iteration for the solution to \eqref{mild}.
 Let $u_0(t \,, x) =  \int_{0}^LG_{t}(x, y) \d y$ 
	for every $(t, x)\in (0, T] \times [0, L]$ and $u_0(0, x) = 1$ for all $x\in  [0, L]$. Define iteratively, for every $n\in\mathbb{Z}_+$,
	\begin{equation}\label{dn1}
		u_{n+1}(t\,,x) := u_0(t\\, x) +
		\int_{0}^t\int_0^LG_{t-r}(x, z)u_n(r\,,z) \,\eta(\d r\,\d z).
	\end{equation}

\begin{lemma}\label{prop:um}
           Let $\{u(t\,, x): (t, x)\in [0, T]\times [0, L]\}$ be the solution to \eqref{PAM} and $\{u_n\}_{n=0}^{\infty}$ 
           be defined in  \eqref{dn1}. Then for all $k\geq 2$, 
           \begin{align}\label{eq:um1}
           c_{T, k}:= \sup_{n\geq 0}\sup_{L\geq 1}\sup_{(t, x)\in [0, T]\times [0, L]}\|u_n(t\,, x)\|_k <\infty
                      \end{align}
           and
                   \begin{align}\label{eq:um2}
         \sup_{L\geq 1}\sup_{(t, x)\in [0, T]\times [0, L]}\|u(t\,, x)\|_k\leq c_{T, k} <\infty.
           \end{align}
\end{lemma}
\begin{proof}
          It is well known that for every $k\geq 2$ and $(t, x)\in [0, T]\times [0, L]$, $u_n(t\,,x)$ converges to
           $u(t\,, x)$ in $L^k(\Omega)$ as $n\to\infty$. Hence \eqref{eq:um2} follows from \eqref{eq:um1} and we only need to 
           prove \eqref{eq:um1}.

           It is clear from the definition of heat kernels in \eqref{heatkernel1},  \eqref{heatkernel2} and \eqref{Pheat} that 
           \begin{align}\label{u(0)}
           \sup_{L\geq 1}\sup_{(t, x)\in [0, T] \times [0, L]} u_0(t, x) \leq 1. 
           \end{align}
           According to \eqref{dn1}, we see from Burkholder's inequality  and Minkowski's inequality that for $k\geq 2$, 
           	\begin{align*}
		\|u_{n+1}(t\,,x) \|_k^2\leq 2 +
		2z_k^2\int_{0}^t\int_0^LG^2_{t-r}(x, z)\|u_n(r\,,z)||_k^2 \,\d z\d r,
	\end{align*}
	where $z_k$ denotes the constant in  Burkholder's inequality. The semigroup property for heat 
	kernel \eqref{semigroup} ensures that
	           	\begin{align}\label{picard}
		\|u_{n+1}(t\,,x) \|_k^2\leq 2 +
		2z_k^2\int_{0}^tG_{2(t-r)}(x, x)\sup_{z\in [0, L]}\|u_n(r\,,z)||_k^2 \,\d r,
	\end{align}
	In the case of Neumann and Dirichlet boundary conditions, we apply the uniform Gaussian upper bound on heat kernel in \eqref{uniformgaussian} to obtain 
         \begin{align*}
		\sup_{z\in [0, L]}\|u_{n+1}(t\,,z) \|_k^2\leq 2 +
		2z_k^2K_T\int_{0}^t\frac{1}{\sqrt{4\pi(t-r)}}\sup_{z\in [0, L]}\|u_n(r\,,z)||_k^2 \,\d r,
	\end{align*}
	where  the constant $K_T$ is defined below  \eqref{uniformgaussian}. Notice that the constants $z_k$ and $K_T$ do not depend on $L\geq 1$. We now apply \cite[Lemma 15 and (56)]{Dalang1999} with $f_n(t)=\sup_{z\in [0, L]}\|u_{n}(t\,,z) \|_k^2$ and $g(s)= \frac{1}{\sqrt{4\pi s}}$ to  obtain \eqref{eq:um1}. 
	
In the case of periodic boundary conditions,  from the expression of periodic heat kernel in \eqref{Pheat}, \eqref{picard} implies that for all $L\geq 1$, 
	           	\begin{align}\label{picard2}
		\|u_{n+1}(t\,,x) \|_k^2\leq 2 +
		2z_k^2\int_{0}^t\sum_{j\in \Z}\bm{p}_{2(t-r)}(j)\sup_{z\in [0, L]}\|u_n(r\,,z)||_k^2 \,\d r.
	\end{align}
	Therefore, in order to apply \cite[Lemma 15 and (56)]{Dalang1999} to conclude the proof,  it suffices to prove that for all $t>0$
	\begin{align}\label{integrable}
	\int_0^t\sum_{j\in \Z}\bm{p}_{r}(j) \, \d r<\infty. 
	\end{align}
	Indeed, by the identity \eqref{poisson}, 
		\begin{align*}
	\int_0^t\sum_{j\in \Z}\bm{p}_{r}(j) \, \d r =	\int_0^t \sum_{n\in \Z}\e^{-2r \pi^2 n^2} \, \d r=  \sum_{n\in \Z}\frac{1-\e^{-2t \pi^2 n^2}}{2\pi^2n^2}  <\infty,
	\end{align*}
	which verifies \eqref{integrable} and hence completes the proof. 
\end{proof}

\begin{remark}\label{stationarity}
In the case of periodic boundary conditions, since the value of the heat kernel $G_{t}(x, y)$ only depends on $x-y$ and the distribution of space-time white noise is shift invariant, we can apply the same arguments as in \cite[Lemma 7.1]{CKNP} to see that for all $t>0$ the process $\{u(t\,, x): x\in [0, L]\}$ is stationary. 
\end{remark}

Recently, Chen et al \cite{CKNP} have proved that the moments of Malliavin derivative of the solution to stochastic heat equation have a Gaussian upper bound; see \cite[Theorem 6.4]{CKNP}. We have the following estimate on the Malliavin derivative of the
 the solution to \eqref{PAM}.

\begin{lemma}\label{derivative:estimate}
	Fix $T >0$. Then $u(t\,, x)\in\bigcap_{k\ge 2}\mathbb{D}^{1, k}$ for all $(t, x)\in [0, T]\times [0, L]$.
	 And  there exists  
	$C_{T, k} >0$ such that for all  $k\geq 2$, $L\geq 1$,  $(t, x)\in [0, T]\times [0, L]$,
	and for almost every $(s\,, y) \in (0\,, t) \times \R$,
	\begin{align}\label{Du(t,x)}
		\|D_{s, y}u(t\,,x)\|_k \leq 
	\begin{cases}
		C_{T, k} \bm{1}_{[0, L]}(y)\,\bm{p}_{t - s}(x - y), & \text{Neumann/Dirichlet case};\\
				C_{T, k} \bm{1}_{[0, L]}(y)\,G_{t - s}(x, y),& \text{periodic case}.
				\end{cases}
	\end{align}
	where $\bm{p}_t(x)$ denotes the heat kernel on $\R$, defined in \eqref{heatkernel}.
\end{lemma}

\begin{proof}
         The proof is similar to that of Theorem 6.4 of Chen et al \cite{CKNP}.  The main difference is that we need control  the moments of Malliavin derivative of $u(t\,, x)$, uniformly for all $L\geq 1$.

	We apply the properties of the divergence operator \cite[Prop.\ 1.3.8]{Nualart}
	in order to deduce from \eqref{dn1} that for almost every $(s\,, y) \in (0\,, t)\times \R$,
	\begin{equation}\label{derivative}
		D_{s, y}u_{n + 1}(t\,,x) = \bm{1}_{[0, L]}(y)G_{t - s}(x, y)u_n(s\,,y) +
		\int_s^t\int_0^LG_{t - r}(x, z)D_{s, y}u_n(r\,,z)\,\eta(\d r\,\d z)
		\quad\text{a.s.}
	\end{equation}
	By induction, we see from \eqref{derivative} that  for all  $n\geq 0$ and $(t, x)\in [0, T] \times [0, L]$,   
	\begin{align}\label{vanish}
	D_{s, y}u_{n }(t\,,x)=0, \, \, \text{a.s.} \quad \text{if $y\not\in [0, L]$}.
	\end{align}
	Moreover, using \eqref{derivative}, \eqref{eq:um1},
	Burkholder's inequality and Minkowski's inequality, for $(s, y)\in (0, t)\times [0, L]$,
	\begin{align}\label{||Du(t,x)||}
		\|D_{s, y}u_{n+1}(t\,,x)\|_k^2& \leq 2c_{T,k}^2\,
		G^2_{t - s}(x, y) + 
		2z^2_k \int_s^t\d r\int_{0}^L\d z\ G^2_{t - r}(x, z)\|D_{s, y}u_n(r\,,z)\|_k^2,
			\end{align}
	where $z_k$ is the constant in Burkholder's inequality. 
 	Let $C_k := (2c_{T,k}^2)\vee(2z_k^2)$. We can iterate \eqref{||Du(t,x)||} to find that  for $(s, y)\in (0, t)\times [0, L]$,
	\begin{align}
		& \|D_{s, y}u_{n + 1}(t\,,x)\|_k^2\nonumber \\
			& \leq \sum_{j=0}^{n}C_k^{j+1}\int_s^t\d r_1
			\int_{0}^L\d z_1
			 \cdots\int_s^{r_{j -1}}\d r_j
			\int_{0}^L\d z_j\  G^2_{t - r_1}(x, z_1) \cdots
			G^2_{r_{j-1} - r_{j}}(z_{j-1}, z_j)G^2_{r_{j} - s}(z_{j}, y). \label{Du_{n+1}}
	\end{align}
	In the case of Neumann/Dirichlet boundary conditions, we apply \eqref{uniformgaussian} to see that 
	\begin{align}
		 \|D_{s, y}u_{n + 1}(t\,,x)\|_k^2
			&\leq \sum_{j=0}^{n}(C_kK_T^2)^{j+1}\int_s^t\d r_1
			\int_\R\d z_1
			 \cdots\int_s^{r_{j -1}}\d r_j\int_\R\d z_j\
			 \nonumber\\
			 & \qquad \qquad \qquad \qquad 
			  \bm{p}^2_{t - r_1}(x- z_1) \cdots
			\bm{p}^2_{r_{j-1} - r_j}(z_{j}- y)\bm{p}^2_{r_{j} - s}(z_{j}- y). \label{eq0}
	\end{align}
	In order to simplify the preceding expression, we need the following two identities
		\begin{align}\label{element}
		\int_{-\infty}^\infty\bm{p}^2_{t -s}(x-y)\bm{p}^2_{s -r}(y-z)\,\d y 
		= \sqrt{\frac{t -r}{4\pi (t -s)(s -r)}}\,\bm{p}^2_{t -r}(x-z),
	\end{align}
	and
	\begin{align}\label{gamma}
	\int_{0<r_j<\cdots<r_1<1}
			\frac{\d r_1\cdots\d r_j}{\sqrt{(1-r_1)\cdots(r_{j-1}-r_j) r_j}} =\frac{\Gamma(1/2)^{j+1}}{\Gamma((j+1)/2)},
	\end{align}
	where $\Gamma$ denotes the gamma function; see \cite[5.14.2]{OLBC10} for \eqref{gamma}.
	We see that
	\eqref{element} and \eqref{gamma} together ensure that
	\begin{align}
		&\int_s^t\d r_1
			\int_\R\d z_1
			 \cdots\int_s^{r_{j -1}}\d r_j
			\int_\R\d z_j\  \bm{p}^2_{t - r_1}(x- z_1) \cdots
			\bm{p}^2_{r_{j-1} - r_j}(z_{j}- y)\bm{p}^2_{r_{j} - s}(z_{j}- y) \nonumber \\
		&\quad = (4\pi)^{-\frac{j}{2}}\,\bm{p}^2_{t-s}(x-y)\int_s^t\d r_1
			\int_s^{r_1}\d r_2\cdots\int_s^{r_{j-1}}\d r_j\,
			\sqrt{\frac{t -s}{(t -r_1)\cdots(r_{j-1}- r_j)(r_j -s)}}  \nonumber \\
		& \quad = \left(\frac{t-s}{4\pi}\right)^{j/2}
			\bm{p}^2_{t-s}(x-y)\int_{0<r_j<\cdots<r_1<1}
			\frac{\d r_1\cdots\d r_j}{\sqrt{(1-r_1)\cdots(r_{j-1}- r_j) r_j}} \nonumber \\
		& \quad = \left(\frac{t-s}{4\pi}\right)^{j/2}
			\frac{\Gamma(1/2)^{j+1}}{\Gamma((j+1)/2)}\,
			\bm{p}^2_{t - s}(x - y). \label{Du_{n+1}'}
	\end{align}
	Hence, in the case of Neumann/Dirichlet boundary conditions, we combine \eqref{eq0}, \eqref{Du_{n+1}'}  and \eqref{vanish} to obtain that for $(s, y)\in (0, t)\times \R$,
	\begin{align}
		\|D_{s, y}u_{n + 1}(t\,,x)\|_k^2 &\leq \bm{1}_{[0, L]}(y)\,
			\bm{p}^2_{t - s}(x - y) \sum_{j = 0}^n (C_kK_T^2)^{j+1}
			\left(\frac{t-s}{4\pi}\right)^{j/2}
			\frac{\Gamma(1/2)^{j+1}}{\Gamma((j+1)/2)} \nonumber \\
		& \leq \bm{1}_{[0, L]}(y)\,\bm{p}^2_{t - s}(x - y)  \sum_{j = 0}^\infty
			\frac{(C_kK_T^2)^{j+1}T^j}{(4\pi)^{j/2}}\frac{\Gamma(1/2)^{j+1}}{\Gamma((j+1)/2)}. \label{uniform1}
	\end{align}
	
	In the case of periodic boundary conditions, analogous to the computations in \eqref{Du_{n+1}'} and \eqref{gamma}, we see from \eqref{Du_{n+1}} and Lemma \ref{subsemi} (compare with identity \eqref{element}) that 
	\begin{align}
		\|D_{s, y}u_{n + 1}(t\,,x)\|_k^2 		& \leq \bm{1}_{[0, L]}(y)\,G^2_{t - s}(x, y)  \sum_{j = 0}^\infty
			\frac{C_k^{j+1}T^j\vartheta(1/(2T\pi))^j}{(4\pi)^{j/2}}\frac{\Gamma(1/2)^{j+1}}{\Gamma((j+1)/2)}. \label{uniform2}
	\end{align}
	Therefore, taking into account \eqref{uniform1} and 
	\eqref{uniform2}, we obtain that there exists a constant $C_{T, k}>0$ such that
   \begin{align}\label{uniform}
		\|D_{s, y}u_{n+1}(t\,,x)\|_k \leq 
	\begin{cases}
		C_{T, k} \bm{1}_{[0, L]}(y)\,\bm{p}_{t - s}(x - y), & \text{Neumann/Dirichlet case};\\
				C_{T, k} \bm{1}_{[0, L]}(y)\,G_{t - s}(x, y),& \text{periodic case}.
				\end{cases}
	\end{align}
	Moreover, \eqref{uniform} yields that 
	\begin{equation}\label{uniformbound}
		\sup_{n \geq 0}\E\left(\|Du_{n}(t\,,x)\|_{\mathcal{H}}^2\right)
			<\infty.
	\end{equation}
	This is because in the Neumann/Dirichlet case, by semigroup property of heat kernel, 
	\begin{equation*}\label{||Du_n||}\begin{split}
		\sup_{n \geq 0}\E\left(\|Du_{n}(t\,,x)\|_{\mathcal{H}}^2\right)
			&\le C_{T,2}^2\int_0^t\d s\int_{-\infty}^\infty\d y\ \bm{p}^2_{t - s}(x - y)=C_{T,2}^2\int_0^t\bm{p}_{2(t-s)}(0)\, \d s = C_{T,2}^2\sqrt{t/\pi}<\infty,
	\end{split}\end{equation*}
	while in the periodic case,
	\begin{equation*}\label{||Du_n||}\begin{split}
		\sup_{n \geq 0}\E\left(\|Du_{n}(t\,,x)\|_{\mathcal{H}}^2\right)
			&\le C_{T,2}^2\int_0^t\d s\int_{0}^L\d y\ G^2_{t - s}(x - y)=C_{T,2}^2\int_0^tG_{2(t-s)}(0,0)\, \d s <\infty,
	\end{split}\end{equation*}
	where we have used 
	\eqref{integrable} in the inequality.
	
	The reminder of the proof follows from a similar approximation argument as in the proof of  \cite[Theorem 6.4]{CKNP}.
	First, 
	we deduce from \eqref{uniformbound}  and \cite[Lemma 1.2.3]{Nualart} that  $u(t\,,x) \in \mathbb{D}^{1,2}$
	and	$Du_{n}(t\,,x)$  converges to $Du(t\,,x)$ in the weak topology of
	$L^2(\Omega\, ; \mathcal{H})$ as $n\to\infty$.
	Then, we use a smooth approximation $\{\psi_\varepsilon\}_{\varepsilon>0}$
	to the identity in $\R_+\times \R$,  and apply Fatou's lemma and
	duality for $L^k$-spaces, in order
	to find that for almost every $(s\,,y) \in (0\,,t) \times \R$ and for all $k\ge 2$,
	\begin{align*}
		\|D_{s,y}u(t\,,x) \|_k & \le  \limsup_{\varepsilon \to 0}
			\left \| \int_0^\infty\d s'\int_{-\infty}^\infty\d y'\, D_{s',y'} u(t\,,x)
			\psi_\varepsilon(s-s', y-y')\right\|_k\\
		& \le  \limsup_{\varepsilon\to 0}
			\sup_{\|G \|_{k/(k - 1)}\le 1}
			\left| \int_0^\infty\d s'\int_{-\infty}^\infty\d y'\,  \E\left[ G D_{s',y'} u(t\,,x) \right]
			\psi_\varepsilon(s-s', y-y') \right|.
	\end{align*}
	Choose and fix a random variable $G\in L^{2}(\Omega)$ such that
	$\| G \|_{k/(k-1)} \le 1$.  Because $Du_{n}(t\,,x)$ converges weakly in $L^2(\Omega\,;\HH)$ 
	to $Du(t\,,x)$ as $n\to\infty$, we can write
	\begin{align*}
		& \left| \int_0^\infty\d s'\int_{-\infty}^\infty\d y'\, \E \left[ G D_{s',y'} u(t\,,x) \right]
			\psi_\varepsilon(s-s', y-y')  \right|  \\
		&\hskip1.5in= \lim _{n\rightarrow \infty} \left|
			\int_0^\infty\d s'\int_{-\infty}^\infty\d y'\,  \E\left[ G D_{s',y'} u_{n}(t\,,x) \right]
			\psi_\varepsilon(s-s', y-y') \right| \\
		&\hskip1.5in\le\limsup_{n\to\infty}
			\int_0^\infty\d s'\int_{-\infty}^\infty\d y'\, \left\| D_{s',y'} u_{n}(t\,,x) \right\|_k
			\psi_\varepsilon(s-s', y-y').  
	\end{align*}
	Now we plug the estimate  \eqref{uniform} in the above line and let $\varepsilon\to 0$ to conclude the proof of \eqref{Du(t,x)}.

	Finally, $u(t\,, x)\in\bigcap_{k\ge 2}\mathbb{D}^{1, k}$ follows immediately from the estimate in \eqref{Du(t,x)}. This completes the proof.
\end{proof}

\section{Asymptotic behavior of the covariance} \label{asym:cov}

In this section, we will analyze the asymptotic behavior of the covariance of the spatial integral of the solution to \eqref{PAM}.

Recall from \eqref{average} that
\[
	\mathcal{S}_{L,t}=\frac 1L \int_0^L \{ u(t\,,x) -\E[u(t\,, x)]\} \,\d x.
\]
The following result provides the asymptotic behavior of the covariance function of the renormalized sequence of processes  $\mathcal{S}_{L,t}$ as $L$ tends to infinity.
\begin{proposition}\label{pr:Cov:asymp}
	For every $t_1,t_2>0$,
	\[
		\lim_{L\to\infty}
		\Cov\left[  \sqrt{L}\, \mathcal{S}_{L,t_1} ~,~
		\sqrt{L}\, \mathcal{S}_{L,t_2}\right] =
		   \int_0^{t_1\wedge t_2} f(s)\, \d s,
	\]
	where the function $f$ is defined in \eqref{f}.
\end{proposition}
In order to prove Proposition \ref{pr:Cov:asymp}, we need the following supporting lemma. 

\begin{lemma}\label{I_0asym}
          Denote for $(t, x)\in (0, \infty) \times [0, L]$,
          \begin{align}\label{I_0}
          \mathcal{I}_0(t\,, x)= \int_0^LG_t(x, y)\, \d y \quad \text{and} \quad \mathcal{I}_0(0, x)= 1. 
          \end{align}
          Then 
          \begin{align}\label{sup}
         \sup_{L\geq 1}\sup_{(t, x)\in [0,\infty)\times [0, L]}\mathcal{I}_0(t\,,x)\leq 1
          \end{align}
          and for all $t_1, t_2>0$,
          \begin{align}\label{eq:I_0}
          \lim_{L\to\infty}\frac{1}{L}\int_0^L \mathcal{I}_0(t_1\,, x)\mathcal{I}_0(t_2\,, x)\, \d x =1. 
          \end{align}       
\end{lemma}
\begin{proof}
          The estimate \eqref{sup} is clear and we need prove \eqref{eq:I_0}.
          By Lemma \ref{hkproperty} (1) and the semigroup property \eqref{semigroup}, we write for all $t_1, t_2>0$,
                    \begin{align*}
          \frac{1}{L}\int_0^L \mathcal{I}_0(t_1\,, x)\mathcal{I}_0(t_2\,, x)\, \d x &=\frac{1}{L}\int_0^L \d x \int_{[0, L]^2}\d y_1\d y_2\, 
          G_{t_1}(x, y_1)          G_{t_2}(x, y_2)\\
          &=\frac{1}{L} \int_{[0, L]^2}
          G_{t_1+t_2}(y_1, y_2) \, \d y_1\d y_2 \to 1, \quad \text{as $L\to \infty$},
          \end{align*}    
          owing to Lemma    \ref{average0}.
\end{proof}
\begin{proof}[Proof of Proposition \ref{pr:Cov:asymp}]
         Using the mild form in \eqref{mild} and Ito's isometry, we write
         \begin{align}
         \Cov\left[  \sqrt{L}\, \mathcal{S}_{L,t_1} ~,~
		\sqrt{L}\, \mathcal{S}_{L,t_2}\right] &= \frac{1}{L}\int_{[0, L]^2} \Cov(u(t_1\,, x)\,, u(t_2\,, y))\, \d x\d y \nonumber\\
		&=\frac{1}{L}\int_{[0, L]^2}  \d x\d y \int_0^{t_1\wedge t_2}\d s\int_0^L\d z\, G_{t_1-s}(x, z)G_{t_2-s}(y, z)\E[u(s\,,z)^2] \nonumber\\
		&=\frac{1}{L} \int_0^{t_1\wedge t_2}\d s\int_0^L\d z\,\mathcal{I}_0(t_1-s\,, z)\mathcal{I}_0(t_2-s\,, z)\,\E[u(s\,,z)^2], \nonumber
         \end{align}
         where the quantity $\mathcal{I}_0$ is defined in \eqref{I_0}.
         Moreover, we have 
         \begin{align*}
                  \Cov\left[  \sqrt{L}\, \mathcal{S}_{L,t_1} ~,~
		\sqrt{L}\, \mathcal{S}_{L,t_2}\right]
		 & = \frac{1}{L} \int_0^{t_1\wedge t_2}\d s\int_0^L\d z\, \left[\mathcal{I}_0(t_1-s\,, z)\mathcal{I}_0(t_2-s\,, z)-1\right]\E[u(s\,,z)^2]\nonumber\\
		 &\quad
		 +   \int_0^{t_1\wedge t_2}\d s\, \frac{1}{L}\int_0^L\d z\,\E[u(s\,,z)^2].
         \end{align*}
         By \eqref{eq:um2}, 
         \begin{align*}
        & \left|  \frac{1}{L} \int_0^{t_1\wedge t_2}\d s\int_0^L\d z\, \left[\mathcal{I}_0(t_1-s\,, z)\mathcal{I}_0(t_2-s\,, z)-1\right]\E[u(s\,,z)^2] \right|\\
         & \quad \leq c_{T,2}^2  \frac{1}{L} \int_0^{t_1\wedge t_2}\d s\int_0^L\d z\, \left[1- \mathcal{I}_0(t_1-s\,, z)\mathcal{I}_0(t_2-s\,, z)\right]
         \to 0 \quad \text{as $L\to\infty$},
         \end{align*}
        thanks to Lemma \ref{I_0asym} and dominated convergence theorem.

         Therefore, applying \eqref{eq:um2} and dominated convergence theorem, we obtain that
                  \begin{align}
         \lim_{L\to\infty}\Cov\left[  \sqrt{L}\, \mathcal{S}_{L,t_1} ~,~
		\sqrt{L}\, \mathcal{S}_{L,t_2}\right] 		&=  \int_0^{t_1\wedge t_2}\d s\lim_{L\to\infty} \frac{1}{L}\int_0^L\d z\,\E[u(s\,,z)^2] \nonumber\\
		&=   \int_0^{t_1\wedge t_2} f(s)\, \d s, \nonumber
         \end{align}
         where the second identity follows from Proposition \ref{pr:Cov:asymp2} below. 
\end{proof}

\begin{proposition}\label{pr:Cov:asymp2}
	For every $t>0$,
	\begin{align}\label{var:f}
	\lim_{L\to\infty} \frac{1}{L}\int_0^L\E[u(t\,,x)^2] \, \d x =f(t),
	\end{align}
        where the function $f$ is defined in \eqref{f}.
\end{proposition}

In order to prove Proposition \ref{pr:Cov:asymp2}, we need the Wiener chaos expansion for the solution to \eqref{PAM}: 
for every $(t, x)\in [0, T] \times [0, L]$,
\begin{align}\label{wiener}
u(t\,, x) =\sum_{k=0}^{\infty} \mathcal{I}_k(t\,, x),
\end{align}
where  $\mathcal{I}_0(t\,, x)$ is defined in \eqref{I_0} and for $k\geq 1$, 
\begin{align}\label{I}
 \mathcal{I}_k(t\,, x)= \int_0^t\int_0^L\eta(\d r_1\, \d z_1)\, G_{t-r_1}(x, z_1)
 \ldots \int_0^{r_{k-1}}\int_0^L\eta(\d r_k\, \d z_k)\, G_{r_{k-1}-r_k}(z_{k-1}, z_k)\mathcal{I}_0(r_k\,, z_k).
\end{align}
Moreover, by multiple Ito's isometry, 
\begin{align}\label{wiener2}
\|u(t\,, x) \|^2_2=\sum_{k=0}^{\infty} \| \mathcal{I}_k(t\,, x)\|^2_2,
\end{align}
where 
\begin{align}\label{I2}
\| \mathcal{I}_k(t\,, x) \|_2^2= \int_0^t\d r_1\int_0^L\d z_1\, G^2_{t-r_1}(x, z_1)
 \ldots \int_0^{r_{k-1}} \d r_{k}\int_0^L\d z_{k}\, G^2_{r_{k-1}-r_k}(z_{k-1}, z_k)\mathcal{I}^2_0(r_k\,, z_k).
\end{align}

\begin{proposition}\label{wienerasym}
Fix $T>0$. Let $\mathcal{I}_k$ be as in \eqref{I}. In the case of Neumann/Dirichlet boundary conditions, for every $k\in \Z_+$,
\begin{align}\label{I_bound}
\sup_{L\geq 1}\sup_{(t, x)\in [0, T]\times [0, L]}\| \mathcal{I}_k(t\,, x) \|_2^2 \leq  \frac{K_T^{2k} 4^{-k/2}T^{k/2}}{\Gamma((k+2)/2)},
\end{align}
where $K_T$ is defined below \eqref{uniformgaussian} and for every $t>0$ and $k\in \Z_+$
\begin{align}\label{I_k:average}
\lim_{L\to\infty} \frac{1}{L}\int_0^L\| \mathcal{I}_k(t\,, x) \|_2^2\, \d x= \frac{(t/4)^{k/2}}{\Gamma((k+2)/2)}.
\end{align}
\end{proposition}
\begin{proof}
From \eqref{sup}, \eqref{I2} and \eqref{uniformgaussian},
          \begin{align}\label{I3}
\| \mathcal{I}_k(t\,, x) \|_2^2&\leq  K_T^{2k} \int_0^t\d r_1\int_{\R}\d z_1\, \bm{p}^2_{t-r_1}(x- z_1)
 \ldots \int_0^{r_{k-1}} \d r_{k}\int_{\R}\d z_{k}\, \bm{p}^2_{r_{k-1}-r_k}(z_{k-1}- z_k) \nonumber\\
 &=  K_T^{2k} (4\pi)^{-k/2}t^{k/2}\int_{0<r_k<\cdots <r_{1}<1}\d r_1\ldots\d r_k \sqrt{\frac{1}{(1-r_1)\times \ldots\times(r_{k-1}-r_k)}}\nonumber\\
 &= K_T^{2k} (4\pi)^{-k/2}t^{k/2}\frac{\Gamma(1/2)^k}{\Gamma((k+2)/2)}\leq \frac{K_T^{2k} 4^{-k/2}T^{k/2}}{\Gamma((k+2)/2)},
\end{align}
where the first equality follows from the elementary identity \eqref{element} and change of variables, and the second one holds by the following identity
\begin{align}\label{gamma3}
\int_{0<r_k<\cdots <r_{1}<1}\d r_1\ldots\d r_k \sqrt{\frac{1}{(1-r_1)\times \ldots\times(r_{k-1}-r_k)}}= \frac{\Gamma(1/2)^k}{\Gamma((k+2)/2)},
\end{align}
see \cite[5.14.1]{OLBC10}. This proves \eqref{I_bound}.

We proceed to prove \eqref{I_k:average}.  By \eqref{I2} and Lemma \ref{hkproperty} (1), (3), 
\begin{align*}
&\frac{1}{L}\int_0^L\| \mathcal{I}_k(t\,, x) \|_2^2\, \d x\\
&\quad= \frac{1}{L}
\int_0^t\d r_1\int_0^L\d z_1\, G_{2(t-r_1)}(z_1, z_1)
  \ldots \int_0^{r_{k-1}} \d r_{k}\int_0^L\d z_{k}\, G^2_{r_{k-1}-r_k}(z_{k-1}, z_k)\mathcal{I}^2_0(r_k\,, z_k)\nonumber\\
  & \quad := J_{k,1}^{(1)}+  J_{k,1}^{(2)}+  J_{k,1}^{(3)},
\end{align*}
where, using the expression of heat kernel for $G_{2(t-r_1)}(z_1, z_1)$ in \eqref{heatkernel1} and \eqref{heatkernel2},
\begin{align*}
J_{k,1}^{(1)}&= 
\frac{1}{L}
\int_0^t\d r_1\, \bm{p}_{2(t-r_1)}(0) \int_0^L\d z_1  \int_0^{r_{1}} \d r_{2}\int_0^L\d z_{2}\, G^2_{r_{1}-r_2}(z_{1}, z_2)\\
&\qquad  \qquad \times
  \ldots \times \int_0^{r_{k-1}} \d r_{k}\int_0^L\d z_{k}\, G^2_{r_{k-1}-r_k}(z_{k-1}, z_k)\mathcal{I}^2_0(r_k\,, z_k),\\
  J_{k,1}^{(2)}&= \frac{1}{L}
\int_0^t\d r_1\,\sum_{n\neq 0}\bm{p}_{2(t-r_1)}(2nL)  \int_0^L\d z_1  \int_0^{r_{1}} \d r_{2}\int_0^L\d z_{2}\, G^2_{r_{1}-r_2}(z_{1}, z_2)\\
&\qquad  \qquad \times
  \ldots \times \int_0^{r_{k-1}} \d r_{k}\int_0^L\d z_{k}\, G^2_{r_{k-1}-r_k}(z_{k-1}, z_k)\mathcal{I}^2_0(r_k\,, z_k),\
,\\
  J_{k,1}^{(3)}&= \pm\frac{1}{L}
  \int_0^t\d r_1\int_0^L\d z_1\,\sum_{n\in \Z}\bm{p}_{2(t-r_1)}(2z_1+ 2nL)   \int_0^{r_{1}} \d r_{2}\int_0^L\d z_{2}\, G^2_{r_{1}-r_2}(z_{1}, z_2)\\
&\qquad  \qquad \times
  \ldots \times \int_0^{r_{k-1}} \d r_{k}\int_0^L\d z_{k}\, G^2_{r_{k-1}-r_k}(z_{k-1}, z_k)\mathcal{I}^2_0(r_k\,, z_k),
\end{align*}
where, in the definition of $  J_{k,1}^{(3)}$, the sign "$+$" ("$-$" respectively) corresponds  to Neumann (Dirichlet) heat kernel. 

Now, we apply \eqref{sup} and \eqref{uniformgaussian} to see that
\begin{align*}
  J_{k,1}^{(2)}&\leq K_T^{2(k-1)}L^{-1}
\int_0^t\d r_1\,\sum_{n\neq 0}\bm{p}_{2(t-r_1)}(2nL) \int_\R\d z_1
 \int_0^{r_{1}} \d r_{2}\int_{\R}\d z_{2}\, \bm{p}^2_{r_{1}-r_2}(z_{1}- z_2)\\
& \qquad \qquad \times
  \ldots \int_0^{r_{k-1}} \d r_{k}\int_{0}^L\d z_{k}\, \bm{p}^2_{r_{k-1}-r_k}(z_{k-1}- z_k)\\
  &= K_T^{2(k-1)}
\int_{0<r_k<\cdots<r_1<t}\d r_1\ldots \d r_k
\,\sum_{n\neq 0}\bm{p}_{2(t-r_1)}(2nL) \,
 \bm{p}_{2(r_{1}-r_2)}(0)\times \ldots \times
 \bm{p}_{2(r_{k-1}-r_k)}(0),
\end{align*}
where we use semigroup property $k-1$ times in the equality. 
Since 
\begin{align}\label{dominated}
\sum_{n\neq 0}\bm{p}_{2(t-r_1)}(2nL) \leq \sum_{n\neq 0}\bm{p}_{2(t-r_1)}(2n) 
\leq \bm{p}_{2(t-r_1)}(0)  \sum_{n\neq 0}\e^{-\frac{n^2}{t}} \quad \text{ for all $L\geq 1$,}
\end{align}
we apply
dominated convergence theorem to obtain $\lim_{L\to \infty}  J_{k,1}^{(2)}=0$.

Moreover, using \eqref{sup}, \eqref{uniformgaussian} and the following two identities
\begin{align}\label{identity}
\bm{p}_t(\sigma x)= \sigma^{-1}\bm{p}_{t/\sigma^2}(x) \quad \text{and} \quad \bm{p}^2_t(x)=\frac{1}{\sqrt{4\pi t}}\bm{p}_{t/2}(x)
\qquad \text{for all $x\in \R$, $\sigma>0$},
\end{align}
we see that 
\begin{align*}
|J_{k,1}^{(3)}|
&\leq K_T^{2(k-1)}L^{-1}
\int_0^t\d r_1 \int_{\R}\d z_1 \,\sum_{n\in \Z}\bm{p}_{2(t-r_1)}(2z_1 +2nL)
 \int_0^{r_{1}} \d r_{2}\int_{\R}\d z_{2}\, \bm{p}^2_{r_{1}-r_2}(z_{1}- z_2)\\
& \qquad \qquad \times
  \ldots \int_0^{r_{k-1}} \d r_{k}\int_0^L\d z_{k}\, \bm{p}^2_{r_{k-1}-r_k}(z_{k-1}- z_k)\\
&=  \frac{R_T}{L}\,\sum_{n\in \Z}
  \int_{0<r_k<\cdots < r_1<t}
  \d r_1\ldots \d r_k
  \frac{1}{\sqrt{(r_1-r_2)\times \ldots\times (r_{k-1}-r_k)}}  \int_{\R}\d z_1\ldots\int_{\R}\d z_{k-1}    \int_0^L\d z_{k}\\
  &\qquad \qquad \times 
    \, \bm{p}_{(t-r_1)/2}(z_1+ nL) \ldots \bm{p}_{(r_{k-2}-r_{k-1})/2}(z_{k-2}- z_{k-1}) 
 \bm{p}_{(r_{k-1}-r_k)/2}(z_{k-1}- z_k)
 \end{align*}
 where $R_T>0$ depends only on $T$.
 Hence, we apply the semigroup property to obtain that 
 \begin{align*}
| J_{k,1}^{(3)}|
&\leq
  \frac{R_T}{L}  \int_{0<r_k<\cdots < r_1<t}
  \d r_1\ldots \d r_k
  \frac{1}{\sqrt{(r_1-r_2)\times \ldots\times (r_{k-1}-r_k)}} \sum_{n\in \Z} \int_0^L\d z_{k} \, \bm{p}_{(t-r_k)/2}( z_k + nL)\\
  & =  \frac{R_T}{L}
  \int_{0<r_k<\cdots < r_1<t}
  \d r_1\ldots \d r_k
  \frac{1}{\sqrt{(r_1-r_2)\times \ldots\times (r_{k-1}-r_k)}},
\end{align*}         
which implies that         $\lim_{L\to \infty}  J_{k,1}^{(3)}=0$.

  The proceeding computation yields that 
  \begin{align*}
\frac{1}{L}\int_0^L\| \mathcal{I}_k(t\,, x) \|_2^2\, \d x&=J_{k,1}^{(1)}+  o(L), \quad \text{as $L\to\infty$}.
\end{align*}

Similarly,  using Lemma \ref{hkproperty} (3) to integrate the integral with respect to $\d z_1$ in the expression of $J_{k,1}^{(1)}$, we can write
\begin{align*}
J_{k,1}^{(1)}= J_{k,2}^{(1)} +J_{k,2}^{(2)} + J_{k,2}^{(3)}, 
\end{align*}
 where 
 \begin{align*}
J_{k,2}^{(1)}&= 
\frac{1}{L}
\int_0^t\d r_1\int_0^{r_1}\d r_2\, \bm{p}_{2(t-r_1)}(0) \bm{p}_{2(r_1-r_2)}(0)  \int_0^L\d z_2  \int_0^{r_{2}} \d r_{3}\int_0^L\d z_{3}\, G^2_{r_{1}-r_2}(z_{2}, z_3)\\
&\qquad  \qquad \times
  \ldots \times \int_0^{r_{k-1}} \d r_{k}\int_0^L\d z_{k}\, G^2_{r_{k-1}-r_k}(z_{k-1}, z_k)\mathcal{I}^2_0(r_k\,, z_k),\\
J_{k,2}^{(2)}&= 
\frac{1}{L}
\int_0^t\d r_1\int_0^{r_1}\d r_2\, \bm{p}_{2(t-r_1)}(0)\sum_{n\neq 0} \bm{p}_{2(r_1-r_2)}(2nL)  \int_0^L\d z_2  \int_0^{r_{2}} \d r_{3}\int_0^L\d z_{3}\, G^2_{r_{1}-r_2}(z_{2}, z_3)\\
&\qquad  \qquad \times
  \ldots \times \int_0^{r_{k-1}} \d r_{k}\int_0^L\d z_{k}\, G^2_{r_{k-1}-r_k}(z_{k-1}, z_k)\mathcal{I}^2_0(r_k\,, z_k),\\
J_{k,2}^{(3)}&= 
\pm
\frac{1}{L}
\int_0^t\d r_1\int_0^{r_1}\d r_2\, \bm{p}_{2(t-r_1)}(0)   \int_0^L\d z_2\, \bm{p}_{2(r_1-r_2)}(2z_2+2nL)  \int_0^{r_{2}} \d r_{3}\int_0^L\d z_{3}\, G^2_{r_{1}-r_2}(z_{2}, z_3)\\
&\qquad  \qquad \times
  \ldots \times \int_0^{r_{k-1}} \d r_{k}\int_0^L\d z_{k}\, G^2_{r_{k-1}-r_k}(z_{k-1}, z_k)\mathcal{I}^2_0(r_k\,, z_k),
\end{align*}
where, in the definition of $  J_{k,2}^{(3)}$, the sign "$+$" ("$-$" respectively) corresponds  to Neumann (Dirichlet) heat kernel. 
 Using the same arguments as before, we obtain that $\lim_{L\to\infty}J_{k,2}^{(2)}=0$ and    $\lim_{L\to\infty}J_{k,2}^{(3)}=0$ and hence 
     \begin{align*}
\frac{1}{L}\int_0^L\| \mathcal{I}_k(t\,, x) \|_2^2\, \d x&=J_{k,2}^{(1)}+  o(L), \quad \text{as $L\to\infty$}.
\end{align*}
     
Therefore,  we can repeat this procedure to  conclude that as $L\to\infty$,
\begin{align*}
\frac{1}{L}\int_0^L\| \mathcal{I}_k(t\,, x) \|_2^2\, \d x& = o(L) +
  \int_{0<r_k<\cdots < r_1<t}
  \d r_1\ldots \d r_k\, \bm{p}_{2(t-r_1)}(0) \ldots  \bm{p}_{2(r_{k-2}-r_{k-1})}(0)\\
  & \qquad 
  \qquad   \qquad \times 
  \frac{1}{L}\int_{0}^L\d z_k\, G_{2(r_{k-1}-r_k)}(z_k, z_k) \mathcal{I}^2_0(r_k\,, z_k).
  \end{align*}
  Moreover, we use the expression of heat kernel for $G_{2(r_{k-1}-r_k)}(z_{k}, z_k)$ in \eqref{heatkernel1} and \eqref{heatkernel2} to decompose the multiple integral above as 
  \begin{align*}
  & \int_{0<r_k<\cdots < r_1<t}
  \d r_1\ldots \d r_k\, \bm{p}_{2(t-r_1)}(0) \ldots \bm{p}_{2(r_{k-1}-r_{k})}(0)
    \frac{1}{L}\int_{0}^L\d z_k\,  \mathcal{I}^2_0(r_k\,, z_k)
  \\
  &  + \int_{0<r_k<\cdots < r_1<t}
  \d r_1\ldots \d r_k\, \bm{p}_{2(t-r_1)}(0)\ldots  \bm{p}_{2(r_{k-2}-r_{k-1})}(0)  \sum_{n\neq0}\bm{p}_{2(r_{k-1}-r_{k})}(2nL)
      \frac{1}{L}\int_{0}^L\d z_k\,  \mathcal{I}^2_0(r_k\,, z_k)
  \\
  &   \pm \int_{0<r_k<\cdots < r_1<t}
  \d r_1\ldots \d r_k\, \bm{p}_{2(t-r_1)}(0) \ldots \bm{p}_{2(r_{k-2}-r_{k-1})}(0)
      \frac{1}{L}\int_{0}^L\d z_k\, \sum_{n\in \Z}\bm{p}_{2(r_{k-1}-r_{k})}(2z_k +2nL)  \mathcal{I}^2_0(r_k\,, z_k)
\end{align*}
where in the last line the sign "$+$" ("$-$" respectively) corresponds  to Neumann (Dirichlet) heat kernel. 
By \eqref{sup}, similar estimate as in \eqref{dominated} and dominated convergence theorem,  the second term above converges to 0 as $L\to\infty$.  Similarly, the third term above also converges to 0 as $L\to\infty$ since
by \eqref{sup} and \eqref{identity}
\begin{align*}
      \frac{1}{L}\int_{0}^L\d z_k\, \sum_{n\in \Z}\bm{p}_{2(r_{k-1}-r_{k})}(2z_k +2nL)  \mathcal{I}^2_0(r_k\,, z_k)
      \leq \frac{1}{2L}\int_{0}^L\d z_k\, \sum_{n\in \Z}\bm{p}_{(r_{k-1}-r_{k})/2}(z_k +nL)=\frac{1}{2L}.
\end{align*}

Therefore, using \eqref{sup}, dominated convergence theorem and \eqref{eq:I_0}, we conclude that
\begin{align}
&\lim_{L\to\infty}\frac{1}{L}\int_0^L\| \mathcal{I}_k(t\,, x) \|_2^2\, \d x\nonumber\\
&\quad = \int_{0<r_k<\cdots < r_1<t}
  \d r_1\ldots \d r_k\, \bm{p}_{2(t-r_1)}(0) \ldots  \bm{p}_{2(r_{k-1}-r_{k})}(0)\lim_{L\to\infty}    \frac{1}{L}\int_{0}^L\d z_k\,  \mathcal{I}^2_0(r_k\,, z_k)
\nonumber\\
&\quad = \int_{0<r_k<\cdots < r_1<t}
  \d r_1\ldots \d r_k\, \bm{p}_{2(t-r_1)}(0) \ldots  \bm{p}_{2(r_{k-1}-r_{k})}(0)\nonumber \\
  &\quad = \frac{(t/4)^{k/2}}{\Gamma((k+2)/2)}, \label{id}
\end{align}
 where the third equality  is due to change of variables and \eqref{gamma3}.  This proves \eqref{I_k:average}.
\end{proof}

\begin{proposition}\label{secondp}
           Let $\mathcal{I}_k$ be as in \eqref{I}. In the case of periodic boundary conditions, for every fixed $t>0$, $\|\mathcal{I}_k(t\,, x)\|_2^2$ does not 
           depend on $x\in [0, L]$ and 
           \begin{align}\label{id2}
           \lim_{L\to\infty}\|\mathcal{I}_k(t\,, x)\|_2^2 = \frac{(t/4)^{k/2}}{\Gamma((k+2)/2)}.
           \end{align}
\end{proposition}
\begin{proof}
         Since $\mathcal{I}_0(t\,, x)\equiv 1$,   the formula \eqref{I2} yields that 
         \begin{align*}
\| \mathcal{I}_k(t\,, x) \|_2^2&= \int_0^t\d r_1\int_0^L\d z_1\, G^2_{t-r_1}(x, z_1)
 \ldots \int_0^{r_{k-1}} \d r_{k}\int_0^L\d z_{k}\, G^2_{r_{k-1}-r_k}(z_{k-1}, z_k)\\
 &= \int_{0<r_k<\cdots < r_1<t}
  \d r_1\ldots \d r_k\, G_{2(t-r_1)}(0, 0) \ldots  G_{2(r_{k-1}-r_{k})}(0, 0),
\end{align*}
where in the second equality, we integrate in the order from $\d z_k$ to $\d z_1$ and use semigroup property of heat kernel and the expression of periodic heat kernel in \eqref{Pheat}. 
Moreover, by dominated convergence theorem, 
 \begin{align*}
           \lim_{L\to\infty}\|\mathcal{I}_k(t\,, x)\|_2^2 &= \int_{0<r_k<\cdots < r_1<t}
  \d r_1\ldots \d r_k\, \lim_{L\to\infty}G_{2(t-r_1)}(0, 0) \ldots  G_{2(r_{k-1}-r_{k})}(0, 0)\\
& = \int_{0<r_k<\cdots < r_1<t}
  \d r_1\ldots \d r_k\, \bm{p}_{2(t-r_1)}(0) \ldots  \bm{p}_{2(r_{k-1}-r_{k})}(0) \\
  &= \frac{(t/4)^{k/2}}{\Gamma((k+2)/2)},
           \end{align*}
           where the last equality holds by \eqref{id}. This proves \eqref{id2}.
\end{proof}

We are now ready to prove Proposition \ref{pr:Cov:asymp2}.
\begin{proof}[Proof of Proposition \ref{pr:Cov:asymp2}]
We first consider the Neumann/Dirichlet case. 
          By \eqref{wiener2} and Fubini's theorem, 
          	\begin{align*}
	 \frac{1}{L}\int_0^L\E[u(t\,,x)^2] \, \d x =\sum_{k=0}^{\infty} 	 \frac{1}{L}\int_0^L \|\mathcal{I}_k(t, x)\|_2^2\, \d x.
	\end{align*}
	Since the series $\sum_{k=0}^{\infty}\frac{K_T^{2k} 4^{-k/2}T^{k/2}}{\Gamma((k+2)/2)}$ converges, by \eqref{I_bound} and dominated convergence theorem, we have
	          	\begin{align*}
	\lim_{L\to\infty} \frac{1}{L}\int_0^L\E[u(t\,,x)^2] \, \d x& =\sum_{k=0}^{\infty} 	 	\lim_{L\to\infty}\frac{1}{L}\int_0^L \|\mathcal{I}_k(t, x)\|_2^2\, \d x=\sum_{k=0}^{\infty} \frac{(t/4)^{k/2}}{\Gamma((k+2)/2)}\\
	&=\sum_{n=1}^{\infty} \frac{(t/4)^{(n-1)/2}}{\Gamma((n+1)/2)}\\
	&= 2\e^{t/4} \int_{-\infty}^{\sqrt{t/2}}\frac{1}{\sqrt{2\pi}}\e^{-y^2/2}\d y
	=f(t),
	\end{align*}
	where the second equality holds by \eqref{I_k:average}, and in the fourth equality 
	 we apply the identity  (see \cite[Lemma 2.3.4]{Che13})
	\begin{align}\label{series}
	\sum_{n=1}^{\infty} \frac{\lambda^{n-1}}{\Gamma((n+1)/2)} = 2\e^{\lambda^2} \int_{-\infty}^{\sqrt{2}\lambda}\frac{1}{\sqrt{2\pi}}\e^{-y^2/2}\d y, \qquad \text{for all $\lambda\geq 0$,}
	\end{align}
         with $\lambda = \sqrt{t/4}$.	This completes the proof of \eqref{var:f} for Neumann/Dirichlet boundary conditions.
         
         Similarly, in the case of periodic boundary conditions, Proposition \ref{secondp} implies that $\E[u(t\,,x)^2]$ does not depend on $x\in [0, L]$ (see also Remark \ref{stationarity}). Hence
         	          	\begin{align*}
	\lim_{L\to\infty} \frac{1}{L}\int_0^L\E[u(t\,,x)^2] \, \d x&=\lim_{L\to\infty} \E[u(t\,,x)^2] = \sum_{k=0}^{\infty} 	 	\lim_{L\to\infty} \|\mathcal{I}_k(t, x)\|_2^2\\
	&=\sum_{n=1}^{\infty} \frac{(t/4)^{(n-1)/2}}{\Gamma((n+1)/2)}
	=f(t),
	\end{align*}
	where in the second equality, we apply dominated convergence theorem in order to exchange the limit and the sum. The proof is complete.
\end{proof}

\begin{remark}
The result of Proposition \ref{pr:Cov:asymp2}
can also be seen from the mild form \eqref{mild}. Indeed, by Ito's isometry and Lemma \ref{hkproperty} (1), (3),
	\begin{align}\label{renew2}
	 \frac{1}{L}\int_0^L\E[u(t\,,x)^2] \, \d x =\frac{1}{L}\int_0^L\mathcal{I}^2_0(t\,, x)\, \d x + \int_0^t \d s\, \frac{1}{L} \int_0^LG_{2(t-s)}(y, y)\E[u(s\,,y)^2]\, \d y
	\end{align}
	For every $t>0$, the existence of $\lim_{L\to\infty}\frac{1}{L}\int_0^L\E[u(t\,,x)^2] \, \d x$ is justified by Propositions \ref{wienerasym} and \ref{secondp}. Moreover, as in the proof of  Proposition \ref{wienerasym}, the dominating term of $G_{2(t-s)}(y, y)$ is $\bm{p}_{2(t-s)}(0)$ as $L\to\infty$. Therefore, by \eqref{eq:um2}, dominated convergence theorem and \eqref{eq:I_0}, we let $L\to \infty$ in \eqref{renew2} to obtain that  $\lim_{L\to\infty}\frac{1}{L}\int_0^L\E[u(t\,,x)^2] \, \d x$ satisfies the 
	renewal equation: for all $t>0$
\begin{align}\label{renew1}
f(t)= 1 + \int_0^t\frac{f(s)}{\sqrt{4\pi(t-s)}}\,\d s,
\end{align}
which admits a unique solution given by the formula in \eqref{f}.
\end{remark}

\begin{remark}
It is clear that the renewal equation \eqref{renew1} also holds for the function $t\mapsto \E[U(t\,, 0)^2]$, where $U$ solves \eqref{PAM2}. In fact, by Ito's isometry and stationarity
\begin{align*}
 \E[U(t\,, 0)^2] &= 1+ \int_0^t\int_{\R}\bm{p}_{t-s}^2(x-y) \E[U(t\,, y)^2]\, \d y\d s\\
 &= 1+ \int_{0}^t\bm{p}_{2(t-s)}(0) \E[U(t\,, 0)^2]\, \d s,
\end{align*}
thanks to semigroup property. We refer to \cite[Chapter 7]{Kho14} for more information on renewal theory related to stochastic heat equation. 
\end{remark}

\section{Proof of Theorems \ref{th:FCLT} and \ref{TVD}}\label{Sec:TVD}

In this section, we will  apply Propositions \ref{pr:MS} and \ref{lemma: NP 6.1.2} to prove Theorems \ref{th:FCLT} and \ref{TVD}.

Recall $\mathcal{S}_{L, t}$ defined in \eqref{average}. Using stochastic Fubini's theorem, we write from \eqref{mild} that 
\begin{equation}\label{S=delta(v)}
	\mathcal{S}_{L,t} = \int_{0}^t\int_0^L v_{L,t}(s\,,y)\,\eta(\d s\,\d y)
	=\delta(v_{L,t})\qquad\text{a.s.,}
\end{equation}
where 
\begin{align}\label{g+v}
	v_{L,t}(s\,,y) &:=L^{-1} \bm{1}_{(0, t)}(s)\bm{1}_{[0, L]}(y)\, u(s\,, y) \int_0^LG_{t-s}(x, y)\, \d x\nonumber\\
	&=L^{-1} \bm{1}_{(0, t)}(s)\bm{1}_{[0, L]}(y) u(s\,, y)\, \mathcal{I}_0(t-s\,, y),
\end{align}
where  in \eqref{g+v} we use Lemma \ref{hkproperty} (1) and the definition of $\mathcal{I}_0(t\,, x)$ in \eqref{I_0}.

The key technical result of this section is the following proposition:
\begin{proposition}\label{Var<Ds,v>}
	For every $T>0$ there exists a real number $A_T>0$ such that
	\[
		\sup_{t,\tau\in[0,T]}
		\Var\left( \left\< D\mathcal{S}_{L,t}\,,v_{L,\tau}\right>_{\HH} \right)
		\le \frac{ A_T}{L^3}\qquad\text{for all $L\ge 1$}.
	\]
\end{proposition}

\begin{proof}
According to Proposition 1.3.2 of \cite{Nualart}, we see from \eqref{S=delta(v)} that 
\begin{equation}\label{DS}
	D_{r,z}\mathcal{S}_{L,t} = \bm{1}_{(0,t)}(r) v_{L,t}(r\,,z) + 
	\bm{1}_{(0,t)}(r) \int_r^t\int_0^LD_{r,z}v_{L,t}(s\,,y)\,\eta(\d s\,\d y).
\end{equation}
Hence,
\begin{align}\label{DS:v:X:Y}
	\left\< D\mathcal{S}_{L,t}\,, v_{L,\tau}\right\>_{\HH} &= 
 \left\< v_{L,t}\,,v_{L,\tau}\right\>_{\HH}  +  \int_0^{\tau}\d r\int_{0}^L\d z\
		v_{L,\tau}(r\,,z)\left( \int_{r}^t\int_0^L D_{r,z} v_{L,t}(s\,,y)\,\eta(\d s\,\d y)\right)\nonumber\\
	&=  \left\< v_{L,t}\,,v_{L,\tau}\right\>_{\HH}  +  \int_{0}^t\int_0^L \left( \int_0^{\tau \wedge s}\d r\int_{0}^L\d z\
		v_{L,\tau}(r\,,z) D_{r,z} v_{L,t}(s\,,y)\right)\,\eta(\d s\,\d y),
\end{align}
where in the second equality we use the stochastic Fubini's theorem. Therefore, 
\begin{align}\label{var1+2}
	\Var\left(\left\< D\mathcal{S}_{L,t}\,, v_{L,\tau}\right\>_{\HH}\right) &\leq 2\left(\Phi^{(1)}_{L, t,\tau} +\Phi^{(2)}_{L, t,\tau}  \right),
\end{align}
where
\begin{align}\label{12}
\Phi^{(1)}_{L, t,\tau} &=\Var\left( \left\< v_{L,t}\,,v_{L,\tau}\right\>_{\HH}  \right),\\
\Phi^{(2)}_{L, t,\tau} &=\Var\left(  \int_{0}^t\int_0^L \left( \int_0^{\tau \wedge s}\d r\int_{0}^L\d z\
		v_{L,\tau}(r\,,z) D_{r,z} v_{L,t}(s\,,y)\right)\,\eta(\d s\,\d y)  \right). \label{21}
		\end{align}

We estimate the two quantities $\Phi^{(1)}_{L, t,\tau}$ and $\Phi^{(2)}_{L, t,\tau}$ separately.  Using the expression in \eqref{g+v}, 
             \begin{align*}
             \Phi^{(1)}_{L, t,\tau} &=\frac{1}{L^{4}}\int_{[0, t\wedge \tau]^2}\d s_1\d s_2 \int_{[0, L]^2}\d y_1\d y_2\, 
             \Cov\left(u^2(s_1\,,y_1)\,, u^2(s_2\,, y_2)\right)\\
             & \qquad \qquad \times \mathcal{I}_0(t-s_1, y_1) \mathcal{I}_0(\tau-s_1, y_1)
              \mathcal{I}_0(t-s_2, y_2) \mathcal{I}_0(\tau-s_2, y_2)
             \end{align*}
             By \eqref{sup}, Poincar\'e inequality \eqref{Poincare:Cov} and the chain rule of Malliavin derivative
              (see \cite[Proposition 3.3.2]{NN}), 
              \begin{align}
             \Phi^{(1)}_{L, t,\tau} &\leq \frac{4}{L^{4}}\int_{[0, t\wedge \tau]^2}\d s_1\d s_2 \int_{[0, L]^2}\d y_1\d y_2\, 
             \int_0^{s_1\wedge s_2}\d r\int_{\R}\d z\, \nonumber\\
             &\qquad \qquad \qquad \times \left\|u(s_1\,, y_1)D_{r, z}u(s_1\,, y_1) \right\|_2
             \left\|u(s_2\,, y_2)D_{r, z}u(s_2\,, y_2) \right\|_2 \nonumber\\
             &\leq \frac{4c_{T,4}^2}{L^{4}}\int_{[0, t\wedge \tau]^2}\d s_1\d s_2 \int_{[0, L]^2}\d y_1\d y_2
             \int_0^{s_1\wedge s_2}\d r\int_{\R}\d z
              \left\|D_{r, z}u(s_1\,, y_1) \right\|_4 \left\|D_{r, z}u(s_2\,, y_2) \right\|_4\label{malliavinnorm}
             \end{align}
             where we have used  H\"{o}lder's inequality and \eqref{eq:um2} in the second inequality.\\

             \noindent{\em{Neumann/Dirichlet case}}. By  Lemma \ref{derivative:estimate},
              \begin{align*}
             \Phi^{(1)}_{L, t,\tau} 
             &\leq \frac{4c_{T,4}^2C_{T,4}^2}{L^{4}}\int_{[0, t\wedge \tau]^2}\d s_1\d s_2 \int_{[0, L]^2}\d y_1\d y_2\, 
             \int_0^{s_1\wedge s_2}\d r\int_{\R}\d z\, \bm{p}_{s_1-r}(y_1-z)\bm{p}_{s_2-r}(y_2-z) \\
              &= \frac{4c_{T,4}^2C_{T,4}^2}{L^{4}}\int_{[0, t\wedge \tau]^2}\d s_1\d s_2 \int_{[0, L]^2}\d y_1\d y_2\, 
             \int_0^{s_1\wedge s_2}\d r\, \bm{p}_{s_1+s_2-2r}(y_1-y_2),
             \end{align*}
             where the equality holds by the semigroup property of heat kernel.  
             Denote 
           \begin{align}\label{I_L}
           I_L(x)= L^{-1}\bm{1}_{[0, L]}(x) \quad \text{ and} \quad 
          \tilde{I}_L(x)=I_L(-x) \qquad \text{ for $x\in \R$}. 
          \end{align}
             We write
                 \begin{align}\label{bound12}
             \Phi^{(1)}_{L, t,\tau} 
                           &\leq  \frac{4c_{T,4}^2C_{T,4}^2}{L^{2}}\int_{[0, t\wedge \tau]^2}\d s_1\d s_2 
             \int_0^{s_1\wedge s_2}\d r\, \left(I_L*\tilde{I}_L*\bm{p}_{s_1+s_2-2r}\right)(0)\nonumber\\
             &\leq  \frac{4c_{T,4}^2C_{T,4}^2}{L^{3}}\int_{[0, t\wedge \tau]^2}\d s_1\d s_2 
             \int_0^{s_1\wedge s_2}\d r \int_{-L}^{L}\bm{p}_{s_1+s_2-2r}(z)\, \d z\nonumber\\
             &\leq  \frac{4T^3c_{T,4}^2C_{T,4}^2}{L^{3}},
             \end{align}
             where, in the second inequality, we use  \cite[(3.17)]{CKNP}.\\
             
             \noindent{\em{Periodic case}}. We see from \eqref{malliavinnorm} and Lemma \ref{derivative:estimate} that
             \begin{align}
             \Phi^{(1)}_{L, t,\tau} 
             &\leq \frac{4c_{T,4}^2C_{T,4}^2}{L^{4}}\int_{[0, t\wedge \tau]^2}\d s_1\d s_2 \int_{[0, L]^2}\d y_1\d y_2\, 
             \int_0^{s_1\wedge s_2}\d r\int_{0}^L\d z\, G_{s_1-r}(y_1, z)G_{s_2-r}(y_2, z) 
            \nonumber\\
              &= \frac{4c_{T,4}^2C_{T,4}^2}{L^{3}}\int_{[0, t\wedge \tau]^2}\d s_1\d s_2 
             \int_0^{s_1\wedge s_2}\d r\nonumber\\
             &\leq  \frac{4T^3c_{T,4}^2C_{T,4}^2}{L^{3}},\label{bound11}
             \end{align}
             where the equality follows from \eqref{int=1}.

             We proceed to estimate $\Phi^{(2)}_{L, t,\tau}$. By Ito's isometry and the expression \eqref{g+v},  we see from \eqref{21} that 
             \begin{align}
             \Phi^{(2)}_{L, t,\tau} &= \int_{0}^t\int_0^L \left\|\int_0^{\tau \wedge s}\d r\int_{0}^L\d z\
		v_{L,\tau}(r\,,z) D_{r,z} v_{L,t}(s\,,y)\right\|_2^2\, \d y\d s\nonumber\\
		 &=\frac{1}{L^4} \int_{0}^t\d s\int_0^L\d y\int_{[0, \tau \wedge s]^2}\d r_1\d r_2\int_{[0, L]^2}\d z_1\d z_2\
		 		\mathcal{I}_0(\tau-r_1, z_1)  \mathcal{I}_0(\tau-r_2, z_2)\mathcal{I}^2_0(t-s, y)\nonumber
 \\
		 & \qquad \qquad 
		\E\left[u(r_1\,,z_1) (D_{r_1,z_1} u(s\,,y))\, u (r_2\,,z_2) D_{r_2,z_2} u(s\,,y)\right] \nonumber \\
		&\leq \frac{1}{L^4} \int_{0}^t\d s\int_0^L\d y\int_{[0, \tau \wedge s]^2}\d r_1\d r_2\int_{[0, L]^2}\d z_1\d z_2\ \nonumber
		 \\ 
		 & \qquad \qquad 
		\|u(r_1\,,z_1)\|_4\, \|D_{r_1,z_1} u(s\,,y)\|_4\, \|u (r_2\,,z_2)\|_4\, \|D_{r_2,z_2} u(s\,,y)\|_4\nonumber\\
		&\leq \frac{c_{T,4}^2}{L^4} \int_{0}^t\d s\int_0^L\d y\int_{[0, \tau \wedge s]^2}\d r_1\d r_2\int_{[0, L]^2}\d z_1\d z_2\
		  \|D_{r_1,z_1} u(s\,,y)\|_4\,  \|D_{r_2,z_2} u(s\,,y)\|_4, \label{2}
             \end{align}
             thanks to \eqref{sup}, H\"{o}lder's inequality and \eqref{eq:um2}.   \\

             \noindent{\em{Neumann/Dirichlet case}}. 
             We apply Lemma \ref{derivative:estimate} again
              in order to obtain that 
                           \begin{align*}
             \Phi^{(2)}_{L, t,\tau} 
             		&\leq \frac{c_{T,4}^2C_{T,4}^2}{L^4} \int_{0}^t\d s\int_0^L\d y\int_{[0, \tau \wedge s]^2}\d r_1\d r_2\int_{[0, L]^2}\d z_1\d z_2\, 
		\bm{p}_{s-r_1}(y-z_1)		\bm{p}_{s-r_2}(y-z_2)\\
		            &\leq \frac{c_{T,4}^2C_{T,4}^2}{L^4} \int_{0}^t\d s\int_{[0, \tau \wedge s]^2}\d r_1\d r_2\int_{[0, L]^2}\d z_1\d z_2\int_{\R}\d y\, 
		\bm{p}_{s-r_1}(y-z_1)		\bm{p}_{s-r_2}(y-z_2)\\
		&= \frac{c_{T,4}^2C_{T,4}^2}{L^4} \int_{0}^t\d s\int_{[0, \tau \wedge s]^2}\d r_1\d r_2\int_{[0, L]^2}\d z_1\d z_2\, 
		\bm{p}_{2s-r_1-r_2}(z_1-z_2),
             \end{align*}
             where we use semigroup property in the equality. We use again the functions $I_L$ and $\tilde{I}_L$
              in \eqref{I_L} and write 
                 \begin{align}
             \Phi^{(2)}_{L, t,\tau} \label{bound2}
             		&\leq\frac{c_{T,4}^2C_{T,4}^2}{L^2} \int_{0}^t\d s\int_{[0, \tau \wedge s]^2}\d r_1\d r_2
				\left(I_L*\tilde{I}_L*\bm{p}_{2s-r_1-r_2}\right)(0)\nonumber\\
			&\leq\frac{c_{T,4}^2C_{T,4}^2}{L^3} \int_{0}^t\d s\int_{[0, \tau \wedge s]^2}\d r_1\d r_2
			\int_{-L}^{L}\bm{p}_{2s-r_1-r_2}(z)\, \d z \nonumber\\
			&\leq\frac{T^3c_{T,4}^2C_{T,4}^2}{L^3},
             \end{align}
             where the second inequality follows from   \cite[(3.17)]{CKNP}.
             \\
             
             \noindent{\em{Periodic case}}. By \eqref{2} and Lemma \ref{derivative:estimate},
              \begin{align}
             \Phi^{(2)}_{L, t,\tau} 
             		&\leq \frac{c_{T,4}^2C_{T,4}^2}{L^4} \int_{0}^t\d s\int_0^L\d y\int_{[0, \tau \wedge s]^2}\d r_1\d r_2\int_{[0, L]^2}\d z_1\d z_2\, 
		G_{s-r_1}(y, z_1)		G_{s-r_2}(y, z_2) \nonumber\\
		           		&=\frac{T^3c_{T,4}^2C_{T,4}^2}{L^3}, \label{bound22}
             \end{align}
             where the second equality is due to  \eqref{int=1}.
             
             Finally, we combine \eqref{var1+2}, \eqref{bound12}, \eqref{bound11}, \eqref{bound2} and \eqref{bound22} to conclude the proof. 
\end{proof}

We are now ready to prove Theorem \ref{TVD}.

\begin{proof}[Proof of Theorem \ref{TVD}]
	We apply  Proposition \ref{Var<Ds,v>} with $t=\tau$ to 
	see that for all $T>0$ there exists $A_T>0$ 
	 such that for all $t\in [0, T]$
	\[
		\Var\left(\left< D\mathcal{S}_{L,t}\,,v_{L,t}\right>_\HH \right)\le 
		\frac{ A_T}{L^3}
		\qquad\text{for all $L\ge1$.}
	\]
	By \eqref{S=delta(v)} and Proposition \ref{pr:MS},
	\begin{align}
		d_{\rm TV}\left( \frac{\mathcal{S}_{L,t}}{\sqrt{\Var(\mathcal{S}_{L,t})}}
			~,~ {\rm N}(0, 1)\right) &\le 2\sqrt{\Var\left< \frac{D\mathcal{S}_{L,t}}{\sqrt{\Var(\mathcal{S}_{L,t})}}
			~,~ \frac{v_{L,t}}{\sqrt{\Var(\mathcal{S}_{L,t})}}\right>_{\HH}
			} \nonumber\\
		&\le 
			\frac{ 2\sqrt{A_T}}{L^{3/2}\Var(\mathcal{S}_{L,t})}\qquad\text{uniformly for all 
			$t\in[0\,,T]$ and $L\ge1$.} \label{tvdfin}
	\end{align}
	 Proposition \ref{pr:Cov:asymp} ensures that
	$\Var(\mathcal{S}_{L,t})\sim \int_0^tf(s)\, \d s/L$ as $L\to\infty$,  where the function $f$ is defined in \eqref{f}. 
	This together with \eqref{tvdfin} implies \eqref{TVDeq}.
\end{proof}

It remains to prove Theorem \ref{th:FCLT}, which consists of 
the weak convergence of the finite-dimensional distributions and tightness.
We establish  the tightness 
in the following proposition.

\begin{proposition}\label{pr:tightness}
	For every $T>0$ and $k\ge2$,
	there exists $\alpha_{T, k}>0$
	such that for all $t_1, t_2 \in [0\,,T]$,
	\begin{align}\label{holder}
				 \left\| \mathcal{S}_{L,t_2}-\mathcal{S}_{L,t_1}
		\right\|_k \le \alpha_{T, k}|t_2 -t_1|^{1/2} L^{-1/2}\qquad\text{uniformly for all $L\ge1$.}
	\end{align}
\end{proposition}
\begin{proof}
          From the express of $\mathcal{S}_{L,t}$ in \eqref{S=delta(v)},
          \begin{align}\label{fubini}
          \mathcal{S}_{L,t}           &=\frac{1}{L} \int_0^t\int_0^L\mathcal{I}_0(t-s\,, y)u(s\,,y)\, \eta(\d s\, \d y),
          \end{align}
 where $\mathcal{I}_0$ is defined in \eqref{I_0}.
          
          Assume $t_1\leq t_2$. 
          
          \textbf{Case 1}: Neumann/periodic boundary conditions.  In this case, we know that $\mathcal{I}_0\equiv 1$. Hence
          by \eqref{fubini}, Burkholder's inequality and Minkowski's inequality, for all $k\geq 2$,
          \begin{align*}
           \left\| \mathcal{S}_{L,t_2}-\mathcal{S}_{L,t_1}
		\right\|_k^2 \leq \frac{z_k^2}{L^2}\int_{t_1}^{t_2}\int_0^L\|u(s\,, y)\|_k^2\, \d y \d s,
          \end{align*}
          where $z_k$ is the constant in Burkholder's inequality. 
          Moreover, we apply Lemma \ref{prop:um} to obtain that for all $L\geq 1$, 
                    \begin{align*}
           \left\| \mathcal{S}_{L,t_2}-\mathcal{S}_{L,t_1}
		\right\|_k^2 &\leq \frac{z_k^2}{L} |t_2-t_1| \sup_{(s, y)\in [0, T]\times [0, L]} \|u(s\,, y)\|_k^2 \\
		&\leq \frac{c_{T, k}^2z_k^2}{L} |t_2-t_1| ,
		          \end{align*}
		          where $c_{T,k}$ is the constant in \eqref{eq:um1}. This proves \eqref{holder} in the case of Neumann/periodic boundary conditions.    
		          
		 \textbf{Case 2}: Dirichlet boundary conditions.  We write from \eqref{fubini} and Burkholder's inequality
		 \begin{align}\label{t_2-t_1}
		 \left\| \mathcal{S}_{L,t_2}-\mathcal{S}_{L,t_1}\right\|_k^2
		& \leq \frac{2z_k^2}{L^2}\Big( \int_{t_1}^{t_2}\int_0^L\mathcal{I}^2_0(t_2-s\,, y)\|u(s\,,y)\|_k^2\, \d y\d s
		\nonumber\\
		 & \qquad \qquad 
		 + \int_{0}^{t_1}\int_0^L(\mathcal{I}_0(t_2-s\,, y)- \mathcal{I}_0(t_1-s\,, y))^2\|u(s\,,y)\|_k^2\, \d y\d s
		 \Big)\nonumber\\
		 &  \leq \frac{2z_k^2c_{T,2}^2}{L^2}\Big( \int_{t_1}^{t_2}\int_0^L\mathcal{I}^2_0(t_2-s\,, y)\, \d y\d s\nonumber\\
		 &\qquad \qquad
		 + \int_{0}^{t_1}\int_0^L(\mathcal{I}_0(t_2-t_1+s\,, y)- \mathcal{I}_0(s\,, y))^2\, \d y\d s
		 \Big)\nonumber\\
		 &:= \frac{2z_k^2c_{T,2}^2}{L^2}(J_1 + J_2),
		 \end{align}
		 where we have used \eqref{eq:um2} in the second inequality.

		 By \eqref{sup}, 
		 \begin{align}\label{J_1}
		 J_1 \leq L|t_2-t_1|.
		 \end{align}
		 In order to estimate $J_2$, we appeal to the representation of Dirichlet heat kernel \eqref{heatkernel3} and write
		 \begin{align*}
		 \mathcal{I}_0(t_2-t_1+s\,, y)- \mathcal{I}_0(s\,, y) 
		 &= \frac{2}{L}\sum_{n=1}^{\infty} \sin(n\pi y/L)\int_0^L\sin(n\pi z/L) \d z 
		 \left[\e^{-\frac{n^2\pi^2 (t_2-t_1+s)}{2L^2}}-\e^{-\frac{n^2\pi^2 s}{2L^2}}\right]\\
		 &=2\sum_{n=1}^{\infty} \sin(n\pi y/L)\frac{1-\cos(n\pi)}{n\pi}
		 \left[\e^{-\frac{n^2\pi^2 (t_2-t_1+s)}{2L^2}}-\e^{-\frac{n^2\pi^2 s}{2L^2}}\right].
		 \end{align*}
		 Now, we apply the $L^2([0, L])$-orthogonality of the functions $y\mapsto \sin(n\pi y/L)$ to obtain that 
		 \begin{align}
		 J_2&= 4\int_0^{t_1}\d s\,
		 \sum_{n=1}^{\infty}\int_0^L \sin^2(n\pi y/L)\d y\, \left[ \frac{1-\cos(n\pi)}{n\pi}\right]^2
		 \left[\e^{-\frac{n^2\pi^2 (t_2-t_1+s)}{2L^2}}-\e^{-\frac{n^2\pi^2 s}{2L^2}}\right]^2 \nonumber\\
		 &= 2L
		 \sum_{n=1}^{\infty} \left[ \frac{1-\cos(n\pi)}{n\pi}\right]^2
		 \left[1-\e^{-\frac{n^2\pi^2 (t_2-t_1)}{2L^2}}\right]^2
		 \int_0^{t_1}\d s\, \e^{-\frac{n^2\pi^2 s}{L^2}}
		 \nonumber\\
		 &\leq 8L
		 \sum_{n=1}^{\infty} \frac{1}{n^2\pi^2}
		 \left[1\wedge \frac{n^2\pi^2 (t_2-t_1)}{2L^2}\right]^2
		  \frac{L^2}{n^2\pi^2}
		 \nonumber\\
		 &\leq \frac8L
		 \sum_{n=1}^{\infty}\left( \frac{L^4}{n^4\pi^4}
		 \wedge|t_2-t_1|^2\right), \nonumber
		 \end{align}
		           where, in the first inequality, we use the fact that $1-\e^{-x} \leq 1\wedge x$ for all $x\geq 0$.
                 Moreover, we have 
                 \begin{align}\label{J_2}
                 J_2&\leq \frac 8L \left(\sum_{n\leq |t_2-t_1|^{-1/2}L/\pi}|t_2-t_1|^2 
                 +\frac{L^4}{\pi^4}\sum_{n> |t_2-t_1|^{-1/2}L/\pi} \frac{1}{n^4}\right)\nonumber\\
                 &\leq \frac 8L 
                 \left(\frac L\pi
                  |t_2-t_1|^{3/2}   
                                 +\frac{L^4}{\pi^4}\int_{ |t_2-t_1|^{-1/2}L/(2\pi) }^{\infty}\frac{1}{y^4}\d y\right)\nonumber\\
                 &=\frac{88}{3\pi}|t_2-t_1|^{3/2}.
                 \end{align}
		 Therefore, we conclude from \eqref{t_2-t_1}, \eqref{J_1} and \eqref{J_2} that for all $k\geq 2$,
		 		 \begin{align*}
		 \left\| \mathcal{S}_{L,t_2}-\mathcal{S}_{L,t_1}\right\|_k^2
		 &\leq  \frac{2z_k^2c_{T,2}^2}{L^2}\left(L|t_2-t_1| + \frac{88}{3\pi}|t_2-t_1|^{3/2}\right),
		 \end{align*}
		 which implies \eqref{holder} in the case of Dirichlet boundary conditions. 
		 
		 The proof is complete. 
\end{proof}

We are now in a position to prove Theorem \ref{th:FCLT}.

\begin{proof}[Proof of Theorem \ref{th:FCLT}]
	The tightness of 
	$\{ \sqrt{L}\,\mathcal{S}_{L,\bullet} \}_{L\ge1 }$ in the space $C[0\,,T]$
	is a direct consequence of Proposition \ref{pr:tightness}. 
	 Therefore, according to  Billingsley \cite{Bil99},  it remains to prove that the finite-dimensional
	distributions of the process $t\mapsto\sqrt{L}\,\mathcal{S}_{L,t}$
	converge to those of $t\mapsto \int_0^t\sqrt{f(s)}\, \d B_s$ as $L\to\infty$, where $f$ is defined in \eqref{f}.
		
	Let us choose and fix  some $T>0$ and $m\ge 1$ points $t_1,\ldots,t_m\in(0\,,T]$.
	Proposition \ref{pr:Cov:asymp} ensures that, for every $i,j=1,\ldots,m$,
	\begin{equation}\label{C}
		\Cov\left( \mathcal{S}_{L,t_i}\,,\mathcal{S}_{L,t_j}\right) 
		\sim \frac{1}{L} \int_0^{t_i\wedge t_j}f(s) \, \d s\qquad\text{as $L\to\infty$}.
	\end{equation}
	Define the following quantities:
	\[
		F_i  :=  \frac{\mathcal{S}_{L,t_i}}{\sqrt{\Var(\mathcal{S}_{L,t_i})}} \quad \text{and} \quad C_{i,j} := \Cov( F_i\,,F_j)
		\qquad \text{for }i, j=1,\ldots,m.
	\]
		We will write $F:=(F_1\,,\ldots,F_m)$,
	and let $G=(G_1\,,\ldots,G_m)$ denote a centered Gaussian random vector
	with covariance matrix $C=(C_{i,j})_{1\le i,j\le m}$.
	
	Recall from \eqref{g+v} the random fields $v_{L,t_1},\ldots,v_{L,t_m}$, and define
	rescaled random fields $V_1,\ldots,V_m$ as follows:
	\[
		V_i := \frac{v_{L,t_i}}{\sqrt{\Var(\mathcal{S}_{L,t_i})}}\qquad 
		\text{for }i=1,\ldots,m.
	\]
	According to \eqref{S=delta(v)},
	$F_i =\delta( V_i)$ for all $i=1,\ldots,m$, and by duality, 
	$\E\<DF_i\,,V_j\>_\HH=\E[F_i\, \delta(V_j)] =C_{i,j}$ for all $i,j=1,\ldots,m$. Therefore,
	Proposition \ref{lemma: NP 6.1.2} implies that
	\[
		\left| \E h(F) - \E h(G) \right| \le\tfrac12 \|h''\|_\infty
		\sqrt{\sum_{i,j=1}^m \Var\< DF_i\,,V_j\>_\HH},
	\]
	for all $h\in C^2(\R^m)$.
	By Proposition \ref{Var<Ds,v>},
	\[
		\Var\<DF_i\,,V_j\>_\HH = 
		\frac{\Var\<D\mathcal{S}_{L,t_i}~,~v_{L,t_j}\>_\HH}{%
		\Var(\mathcal{S}_{L,t_i})\Var(\mathcal{S}_{L,t_j})}
		\le \frac{A_T}{L^3\Var(\mathcal{S}_{L,t_i})\Var(\mathcal{S}_{L,t_j})},
	\]
	which together with \eqref{C} implies that
	\begin{equation}\label{h(F)}
		\lim_{L\to\infty}\left| \E h(F) - \E h(G) \right| =0, \qquad \text{for all $h\in C^2(\R^m)$.} 
	\end{equation}
	
	On the other hand,  owing to \eqref{C}, 
	as $L\to\infty$,
	\begin{align*}
	C_{i,j} = \frac{\Cov\left( \mathcal{S}_{L,t_i}\,,\mathcal{S}_{L,t_j}\right) }{\sqrt{\Var(\mathcal{S}_{L,t_i})}\sqrt{\Var(\mathcal{S}_{L,t_j})}} \to \frac{\int_0^{t_i\wedge t_j} f(s)\, \d s}{\sqrt{\int_0^{t_i} f(s)\, \d s\times \int_0^{ t_j} f(s)\, \d s}},
	\end{align*}
	which yields that as $L\to \infty$, 
the random vector 	$G$ converges weakly to 
\begin{align}
\left(\frac{\int_0^{t_1}\sqrt{f(s)}\, \d B_s}{\sqrt{\int_0^{t_1}f(s) \, \d s}}, \ldots, \frac{\int_0^{t_m}\sqrt{f(s)}\, \d B_s}{\sqrt{\int_0^{t_m}f(s) \, \d s}}\right). \label{vector}
\end{align}

	Therefore, it follows from \eqref{h(F)} that $F$ converges weakly to the random vector in \eqref{vector}
	as $L\to\infty$. One more appeal to \eqref{C}
	shows that as $L\to\infty$,
	\[
		\sqrt{L}
		\left( \frac{\mathcal{S}_{L,t_1}}{\sqrt{\int_0^{t_1}f(s) \, \d s}}\,,\ldots,
		\frac{\mathcal{S}_{L,t_m}}{\sqrt{\int_0^{t_m}f(s) \, \d s}}\right)
		\rightarrow
\left(\frac{\int_0^{t_1}\sqrt{f(s)}\, \d B_s}{\sqrt{\int_0^{t_1}f(s) \, \d s}}, \ldots, \frac{\int_0^{t_m}\sqrt{f(s)}\, \d B_s}{\sqrt{\int_0^{t_m}f(s) \, \d s}}\right)
	\]
	in distribution.  This completes the proof.
\end{proof}

\appendix
\section{Appendix}

We include in this section a few properties of heat kernel with Neumann, Dirichlet or periodic boundary conditions that are used in this paper, some of which could also be found in \cite{Walsh} and \cite{BMS95}.  

Denote the heat kernel on $\R$ as
\begin{align}\label{heatkernel}
\bm{p}_{t}(x)= \frac{1}{\sqrt{2\pi t}}\e^{-\frac{x^2}{2t}}, \quad t>0, \,x\in \R.
\end{align}
Recall that we have the following formulas for the heat kernels on $[0, L]$: for all $t>0$ and $x, y \in [0, L]$, in the case of  Neumann boundary conditions ($\partial_xu(t\,,L)=\partial_xu(t\,,0)=0$),
\begin{align}\label{heatkernel1}
G_{t}(x, y) = \sum_{n\in \Z}\left( \bm{p}_{t}(x-y+2nL) + \bm{p}_{t}(x+y+2nL)\right),
\end{align}
or equivalently, 
\begin{align}\label{heatkernel4}
G_{t}(x, y) =L^{-1/2} + \frac{2}{L} \sum_{n=1}^{\infty} \cos(n\pi x/L)\cos(n\pi y/L)\e^{-\frac{n^2\pi^2t}{2L^2}};
\end{align}
in the case of Dirichlet boundary conditions ($u(t\,,L)=u(t\,,0)=0$),
\begin{align}\label{heatkernel2}
G_{t}(x, y) = \sum_{n\in \Z}\left( \bm{p}_{t}(x-y+2nL) - \bm{p}_{t}(x+y+2nL)\right), 
\end{align}
or equivalently, 
\begin{align}\label{heatkernel3}
G_{t}(x, y) = \frac{2}{L}\sum_{n=1}^{\infty} \sin(n\pi x/L)\sin(n\pi y/L)\e^{-\frac{n^2\pi^2 t}{2L^2}};
\end{align}
and in the case of periodic boundary conditions ($u(t\,,L)=u(t\,,0)$, $\partial_xu(t\,,L)=\partial_xu(t\,,0)$)
\begin{align}\label{Pheat}
G_{t}(x, y) = \sum_{n\in \Z} \bm{p}_{t}(x-y+nL).
\end{align}


\begin{lemma}\label{hkproperty}
         \begin{itemize}
         \item [(1)] Symmetry. $G_t(x, y)=G_t(y, x)$ for all $t>0$, $x, y\in [0, L]$.
         \item [(2)] In the case of Neumann and periodic heat kernel, for all $t >0$ and $x\in [0, L]$, 
         \begin{align}\label{int=1}
         \int_0^LG_t(x, y)\, \d y=1.
         \end{align}
         \item  [(3)] Semigroup property.  For all $t, s>0$ and $x, y \in [0, L]$, 
         \begin{align}\label{semigroup}
         \int_0^L G_t(x, z)G_s(z, y)\, \d z = G_{t+s}(x, y). 
         \end{align}
         \item [(4)] In the case of Neumann and Dirichlet heat kernel, for every $t>0$ and $x, y\in [0, L]$,
         \begin{align}\label{gaussian}
         G_t(x, y) \leq   \bm{p}_{t}(x-y) \left(4+  \frac{4}{1- \e^{-L^2/t}}\right).
         \end{align}
         As a consequence, for all $t\in (0, T]$, $L\geq 1$ and $x, y \in [0, L]$,
         \begin{align}\label{uniformgaussian}
         G_t(x, y) \leq   K_T\, \bm{p}_{t}(x-y),
                  \end{align}
         where $K_T= 4+  \frac{4}{1- \e^{-1/T}}$.
                  \end{itemize}
\end{lemma}
\begin{proof}
          The properties (1)-(3) are obvious and we need prove (4).  Since the Dirichlet heat kernel is less than the Neumann heat kernel (compare \eqref{heatkernel1} and \eqref{heatkernel2}), it suffices to prove  (4) for Neumann heat kernel. 
          
            For every $t>0$, $n\in \Z\setminus\{-1, 1\}$ and $x, y \in [0, L]$, we write
          \begin{align*}
           \bm{p}_{t}(x-y+2nL) &= \bm{p}_{t}(x-y)\, \e^{-\frac{4n^2L^2 + 4(x-y)nL}{2t}} \\
            &\leq  \bm{p}_{t}(x-y)\, \e^{-\frac{4n^2L^2 -4|n|L^2}{2t}} \leq   \bm{p}_{t}(x-y)\e^{-\frac{|n|L^2}{t}}.
          \end{align*}
          And for $n\in \{-1, 1\}$, we have $\bm{p}_{t}(x-y+2nL) \leq \bm{p}_{t}(x-y)$ for all $t>0$ and $x, y \in [0, L]$.
          Hence, for every $t>0$ and  $x, y \in [0, L]$,
           \begin{align}\label{first}
           \sum_{n\in \Z}\bm{p}_{t}(x-y+2nL) 
            &\leq    \bm{p}_{t}(x-y)\left(1+2\sum_{n=0}^{\infty}\e^{-\frac{nL^2}{t}} \right)=     \bm{p}_{t}(x-y)\left(1+ \frac{2}{1- \e^{-L^2/t}}\right).
          \end{align}
          Similarly, for  every $t>0$, $n\in \Z\setminus\{-1,1 ,-2, 2\}$ and $x, y \in [0, L]$, we have 
          \begin{align*}
           \bm{p}_{t}(x+y+2nL) &= \bm{p}_{t}(x-y)\,\e^{-\frac{4n^2L^2 + 4(x+y)nL + 4xy}{2t}} \\
            &\leq  \bm{p}_{t}(x-y)\,\e^{-\frac{4n^2L^2 -8|n|L^2}{2t}} \leq   \bm{p}_{t}(x-y)\e^{-\frac{|n|L^2}{t}},
          \end{align*}
          Clearly,  $\bm{p}_{t}(x+y+2nL) \leq \bm{p}_{t}(x-y)$ for $n\in\{1, -2, 2\}$ for all $t>0$ and $x, y \in [0, L]$.
          Moreover, we observe that 
          \begin{align*} 
          \bm{p}_{t}(x+y-2L) =   \bm{p}_{t}(x-y)\, \e^{-\frac{2(L-x)(L-y)}{t}} \leq    \bm{p}_{t}(x-y).
          \end{align*}
          The proceeding estimates together imply that for every $t>0$ and  $x, y \in [0, L]$,
          \begin{align}\label{second}
          \sum_{n\in \Z}\bm{p}_{t}(x+y+2nL) 
            &\leq     \bm{p}_{t}(x-y) \left(3+ 2\sum_{n=0}^{\infty}\e^{-\frac{nL^2}{t}} \right) \nonumber\\
            &=  \bm{p}_{t}(x-y) \left(3+  \frac{2}{1- \e^{-L^2/t}}\right).
          \end{align}
          Therefore, we combine \eqref{first} and \eqref{second} to obtain \eqref{gaussian}. 
          Finally, \eqref{uniformgaussian} is an immediate consequence of \eqref{gaussian}.
 \end{proof}

\begin{remark}\label{Dirichlet}
\begin{itemize}
\item [(1)]
There is not a uniform Gaussian lower  bound for the Dirichlet heat kernel. If there is a constant $C_T>0$ such that for all $t\in (0, T]$, $L\geq 1$ and $x, y\in [0, L]$
\begin{align*}
C_T\, \bm{p}_{t}(x-y) \leq  \sum_{n\in \Z}\left( \bm{p}_{t}(x-y+2nL) - \bm{p}_{t}(x+y+2nL)\right),
\end{align*}
then let $x=y=L$ in the above inequality to obtain 
\begin{align*}
C_T\, \bm{p}_{t}(0) \leq  \sum_{n\in \Z\setminus \{0\}} \bm{p}_{t}(2nL) - \sum_{n\in \Z\setminus\{-1\}}\bm{p}_{t}(2L+2nL).
\end{align*}
This gives a contradiction by letting $L\to\infty$.
\item [(2)]
The periodic heat kernel given by \eqref{Pheat}
does not have a uniform Gaussian upper bound as in \eqref{uniformgaussian}. To see this, suppose that there exists $C_T>0$ such that 
\begin{align*}
 \sum_{n\in \Z} \bm{p}_{t}(x-y+nL) \leq C_T\, \bm{p}_{t}(x-y), \quad\text{for all  $0< t\leq T$, $L\geq 1$ and for all$ \,x, y \in [0, L]$.}
\end{align*}
Letting $x=L$, $y=0$ and choosing $n=-1$, it leads to $\bm{p}_t(0)\leq C_T\, \bm{p}_t(L)$ for all $L\geq 1$, which gives a contradiction by letting $L\to \infty$.  However, the periodic heat kernel possesses the following sub-semigroup property. 

\end{itemize}
\end{remark}

\begin{lemma}\label{subsemi}
In the case of periodic heat kernel, for all $s, t\in (0, T]$, $L\geq 1$ and $x, y\in [0, L]$
         \begin{align}\label{squaresub}
         \int_0^LG_t^2(x, z)G_s^2(z, y)\, \d z \leq \sqrt{\frac{t+s}{4  \pi st }} \vartheta\left(\frac{1}{2T\pi}\right)\,G_{t+s}^2(x, y),
         \end{align}
         where $\vartheta$ denotes the theta function given by $\vartheta(r) = \sum_{n\in \Z}\e^{-\pi n^2 r}, r>0$.
\end{lemma}
\begin{proof}
         According to the formula in \eqref{Pheat},
          \begin{align*}
          G_t^2(x, z)G_s^2(z, y)&=\sum_{m_1,m_2, n_1,n_2\in \Z} \bm{p}_t(x-z+m_1L) \bm{p}_t(x-z+m_2L) \bm{p}_s(z-y+n_1L) \bm{p}_s(z-y+n_2L)\\
          &=4\sum_{m_1,m_2, n_1,n_2\in \Z} \bm{p}_{2t}(2(x-z)+(m_1+m_2)L) \bm{p}_{2s}(2(z-y)+(n_1+n_2)L)\\
          &\qquad \qquad\qquad \qquad  \times \bm{p}_{2t}((m_1-m_2)L)\bm{p}_{2s}((n_1-n_2)L),
          \end{align*}
          where in the second equation we use the following elementary identity
          \begin{equation} \label{E3}
\pmb{p}_\sigma(x)  \pmb{p}_\sigma(y)= 2 \pmb{p}_{2\sigma}(x+y)  \pmb{p}_{2\sigma}(x-y) , \qquad \sigma>0, \, x,y \in \R,
\end{equation}
           Using change of variables ($m_1-m_2=k$, $n_1-n_2=j$) and the identity
                      \begin{align}\label{scale}
\bm{p}_t(\sigma x) = \sigma^{-1}\bm{p}_{t/\sigma^2}(x), \quad \text{for all $x \in \R$ and $t, \,\sigma>0$},
\end{align}
          it yields that 
                     \begin{align*}
          &G_t^2(x, z)G_s^2(z, y)\\
          &\quad=\sum_{k,m_2, j,n_2\in \Z} \bm{p}_{t/2}(x-z+(m_2 +k/2)L) \bm{p}_{s/2}(z-y+(n_2+j/2)L)
           \bm{p}_{2t}(kL)\bm{p}_{2s}(jL)\\
             &\quad=\sum_{k,\ell, j,n_2\in \Z} \bm{p}_{t/2}(x+(\ell +k/2)L- (z+n_2L)) \bm{p}_{s/2}(z+n_2L-y+(j/2)L)
           \bm{p}_{2t}(kL)\bm{p}_{2s}(jL),
          \end{align*}
where the last equality follows from change of variable ($m_2=\ell -n_2$). Therefore, 
\begin{align*}
        & \int_0^LG_t^2(x, z)G_s^2(z, y)\, \d z \\
        &=\sum_{k,\ell, j\in \Z}           \bm{p}_{2t}(kL)\bm{p}_{2s}(jL) \sum_{n_2\in \Z} \int_0^L\bm{p}_{t/2}(x+(\ell +k/2)L- (z+n_2L)) \bm{p}_{s/2}(z+n_2L-y+(j/2)L)\d z\\
        &=\sum_{k,\ell, j\in \Z}           \bm{p}_{2t}(kL)\bm{p}_{2s}(jL) \sum_{n_2\in \Z} \int_{n_2L}^{(n_2+1)L}\bm{p}_{t/2}(x+(\ell +k/2)L- z) \bm{p}_{s/2}(z-y+(j/2)L)\d z\\
        &=\sum_{k,\ell, j\in \Z}           \bm{p}_{2t}(kL)\bm{p}_{2s}(jL)  \int_\R\bm{p}_{t/2}(x+(\ell +k/2)L- z) \bm{p}_{s/2}(z-y+(j/2)L)\d z\\
        &=\sum_{k,\ell, j\in \Z}           \bm{p}_{2t}(kL)\bm{p}_{2s}(jL)  \bm{p}_{(t+s)/2}(x-y+\ell L +(k+j)L/2),
         \end{align*}
         thanks to the semigroup property of heat kernel. 
         
         Using again the identity \eqref{scale}, we obtain that 
         \begin{align*}
         \int_0^LG_t^2(x, z)G_s^2(z, y)\, \d z    &=2\sum_{k,\ell, j\in \Z}           \bm{p}_{2t}(kL)\bm{p}_{2s}(jL)  \bm{p}_{2(t+s)}(2(x-y)+2\ell L +(k+j)L)\\
         &=2\sum_{m,\ell, j\in \Z}    \bm{p}_{2(t+s)}(2(x-y)+2\ell L +mL)       \bm{p}_{2t}((m-j)L)\bm{p}_{2s}(jL)  \\
         &=\sum_{m,\ell, j\in \Z}    \bm{p}_{t+s}(x-y+\ell L) \bm{p}_{t+s}(x-y+\ell L +mL)  
         \frac{    \bm{p}_{2t}((m-j)L)\bm{p}_{2s}(jL)}{ \bm{p}_{2(t+s)}(mL)}
         \end{align*}
         where we use the identity \eqref{E3} in the third equality.  Now we apply the following identity 
         \begin{equation*}\label{PPPP}
	\frac{\bm{p}_{t-s}(a)\bm{p}_s(b)}{\bm{p}_t(a+b)} = 
	\bm{p}_{s(t-s)/t}\left( b - \frac st (a+b)\right)
	\quad\text{for all $0<s<t$ and $a,b\in\R$}.
\end{equation*}
in order to deduce  that 
         \begin{align}\label{eq1}
         &\int_0^LG_t^2(x, z)G_s^2(z, y)\, \d z   \nonumber\\
                   &\quad =\sum_{m,\ell, j\in \Z}    \bm{p}_{t+s}(x-y+\ell L) \bm{p}_{t+s}(x-y+\ell L +mL) 
         \bm{p}_{2ts/(t+s)}\left(jL - \frac{s}{t+s}mL\right). 
         \end{align}

By Poisson summation formula (see \cite[Theorem 3.1, Chap. 3]{Stein03}), we have
\begin{align}\label{poisson}
\sum_{j\in \Z}\bm{p}_{\sigma}(j + a) = \sum_{n\in \Z}\e^{-2\sigma \pi^2 n^2}\e^{2\pi i n a}, \quad \text{for all $\sigma >0$ and $a\in \R$.}
\end{align}
Hence, using \eqref{scale},
\begin{align}\label{bound}
\sum_{j\in \Z}\bm{p}_{2ts/(t+s)}\left(jL - \frac{s}{t+s}mL\right)&=\frac{1}{L}\sum_{j\in \Z}\bm{p}_{2ts/[(t+s)L^2]}\left(j- \frac{s}{t+s}m\right) \nonumber\\
&=\frac{1}{L} \sum_{n\in \Z}\e^{-\frac{4st \pi^2 n^2}{(t+s)L^2}}\e^{-\frac{2\pi i nsm}{t+s} }\nonumber\\
&\leq \frac{1}{L} \sum_{n\in \Z}\e^{-\frac{4st \pi^2 n^2}{(t+s)L^2}}= \frac{1}{L} \vartheta\left(\frac{4st\pi} {(t+s)L^2}\right),
\end{align}
where $\vartheta(r) = \sum_{n\in \Z}\e^{-\pi n^2 r}, r>0$ denotes the theta function.  
By \cite[Theorem 3.2, Chap. 3]{Stein03}, we have for all $s,t \in (0, T]$ and $L\geq1$,
\begin{align}\label{bound1}
 \frac{1}{L} \vartheta\left(\frac{4st\pi} {(t+s)L^2}\right)&=\sqrt{\frac{t+s}{4\pi st } }\, \vartheta\left(\frac{(t+s)L^2}{4st\pi} \right)
 \leq \sqrt{\frac{t+s}{4\pi st } }\, \vartheta\left(\frac{1}{2T\pi}\right),
\end{align}
where the inequality holds since the theta function is decreasing.

Finally, we combine \eqref{eq1}, \eqref{bound} and \eqref{bound1} to obtain 
\begin{align*}
         \int_0^LG_t^2(x, z)G_s^2(z, y)\, \d z&\leq \sqrt{\frac{t+s}{4\pi st }}  \vartheta\left(\frac{1}{2T\pi}\right)\sum_{m, \ell\in \Z}    \bm{p}_{t+s}(x-y+\ell L) \bm{p}_{t+s}(x-y+\ell L +mL) \\
         &= \sqrt{\frac{t+s}{4  \pi st }} \vartheta\left(\frac{1}{2T\pi}\right) \,G_{t+s}^2(x, y).
         \end{align*}
The proof is complete.          
\end{proof}

\begin{lemma}\label{average0}
          For all $t>0$,
          \begin{align}\label{eq:av}
          \lim_{L\to\infty}\frac{1}{L}\int_0^L\int_0^L G_t(x, y)\, \d x\d y =1.
          \end{align}
\end{lemma}
\begin{proof}
          It is clear that \eqref{eq:av} holds for the heat kernel with Neumann and periodic boundary conditions; see \eqref{int=1}. Let us check it is 
          also true for Dirichlet heat kernel.  Recall the alternative representation of Dirichlet heat kernel in \eqref{heatkernel3}.  We write for all $t>0$,
          \begin{align*}
          \frac{1}{L}\int_0^L\int_0^L G_t(x, y)\, \d x\d y &=        
             \frac{2}{L^2}\sum_{n=1}^{\infty}
             \int_0^L\sin(n\pi x/L)\, \d x              \int_0^L\sin(n\pi y/L)\, \d y
           \,  \e^{-\frac{n^2\pi^2t}{2L^2}}\\
           &=        
             \frac{2}{\pi^2}\sum_{n=1}^{\infty}
           \frac{ (1-\cos(n\pi))^2}{n^2}
            \,  \e^{-\frac{n^2\pi^2t}{2L^2}}\\
           &  =           \frac{8}{\pi^2}\sum_{k=1}^{\infty}
           \frac{ 1}{(2k-1)^2}
            \,  \e^{-\frac{(2k-1)^2\pi^2t}{2L^2}}.
          \end{align*}
          By dominated convergence theorem, for all $t>0$,
                    \begin{align*}
          \lim_{L\to\infty}\frac{1}{L}\int_0^L\int_0^L G_t(x, y)\, \d x\d y 
                     &  =           \frac{8}{\pi^2}\sum_{k=1}^{\infty}
           \frac{ 1}{(2k-1)^2}
            \,  \lim_{L\to\infty}\e^{-\frac{(2k-1)^2\pi^2t}{2L^2}}\\
            &=           \frac{8}{\pi^2}\sum_{k=1}^{\infty}
           \frac{ 1}{(2k-1)^2}=1,
          \end{align*}
         where the last identity follows from the fact $\sum_{k=1}^{\infty}
           \frac{ 1}{(2k-1)^2}=  \frac{\pi^2}{8}$. This proves \eqref{eq:av}.
\end{proof}

\end{document}